\documentclass{amsart}
\usepackage{amssymb}
\usepackage{mathdots}
\usepackage[...]{youngtab}
\usepackage{tikz}
\usetikzlibrary{shapes,snakes}
\usepackage{framed}
\usepackage{tikz-cd}
\usepackage{amsmath}
\usepackage{txfonts}

\input xy
\xyoption{all}

\usepackage{amsfonts}
\usepackage{mathrsfs}

\usepackage{latexsym}
\usepackage{graphicx}
\usepackage{amscd,amssymb,amsmath,amsbsy,amsthm}
\usepackage[all]{xy}

\usepackage{xcolor}
\definecolor{darkred}{RGB}{139,0,0}
\usepackage{hyperref}
\hypersetup{colorlinks,linkcolor={darkred},citecolor={darkred},urlcolor={darkred}} 
\usepackage{verbatim}

\usepackage{tikz-cd}
\usepackage{amsmath}
\usepackage{amsfonts}
\usepackage{amssymb}
\usepackage{pdfpages}
\usepackage{tikz}
\usepackage{asymptote}
\usetikzlibrary{decorations.markings,arrows}
\usepackage[all]{xy}
\usepackage{yfonts}
\usepackage[...]{youngtab}

\usepackage{mathrsfs}
\usepackage{latexsym}
\usepackage{graphicx}
\usepackage{amscd,amssymb,amsmath,amsbsy,amsthm}
\usepackage[all]{xy}

\newtheorem{theorem}{Theorem}[section]
\newtheorem{lemma}[theorem]{Lemma}
\newtheorem{corollary}[theorem]{Corollary}
\newtheorem{prop}[theorem]{Proposition}

\theoremstyle{definition}
\newtheorem{defn}[theorem]{Definition}
\newtheorem{ex}[theorem]{Example}

\DeclareMathOperator{\id}{\mathbf{id}}

\DeclareMathOperator{\FF}{\mathbb{F}}
\DeclareMathOperator{\NN}{\mathbb{N}}
\DeclareMathOperator{\ZZ}{\mathbb{Z}}
\DeclareMathOperator{\QQ}{\mathbb{Q}}
\DeclareMathOperator{\RR}{\mathbb{R}}
\DeclareMathOperator{\CC}{\mathbb{C}}
\DeclareMathOperator{\PP}{\mathbb{P}}

\DeclareMathOperator{\KK}{\mathbb{K}}
\DeclareMathOperator{\TT}{\mathbb{T}}

\DeclareMathOperator{\rad}{\mathbf{rad}}
\DeclareMathOperator{\cG}{\mathcal{G}}
\DeclareMathOperator{\cZ}{\mathcal{Z}}
\DeclareMathOperator{\cN}{\mathcal{N}}

\DeclareMathOperator{\Mod}{\mathbf{Mod}}

\DeclareMathOperator{\rep}{\mathbf{rep}}

\DeclareMathOperator{\cO}{\mathcal{O}}
\DeclareMathOperator{\lcm}{\mathbf{lcm}}

\DeclareMathOperator{\fm}{\mathfrak{m}}
\DeclareMathOperator{\fb}{\mathfrak{b}}
\DeclareMathOperator{\fg}{\mathfrak{g}}
\DeclareMathOperator{\frp}{\mathfrak{p}}
\DeclareMathOperator{\bd}{\mathbf{d}}
\DeclareMathOperator{\sO}{\mathcal{O}}
\DeclareMathOperator{\SL}{\mathbf{SL}}
\DeclareMathOperator{\SU}{\mathbf{SU}}
\DeclareMathOperator{\GL}{\mathbf{GL}}
\DeclareMathOperator{\su}{\mathfrak{su}}
\DeclareMathOperator{\sln}{\mathfrak{sl}}
\DeclareMathOperator{\gl}{\mathfrak{gl}}

\DeclareMathOperator{\Spec}{\mathbf{Spec}}

\DeclareMathOperator{\Mat}{\mathbf{Mat}}
\DeclareMathOperator{\Hom}{\mathbf{Hom}}

\DeclareMathOperator{\cA}{\mathbf{\mathcal{A}}}
\DeclareMathOperator{\bA}{\mathbf{A}}

\newcommand{\la}{\langle}
\newcommand{\ra}{\rangle}

\newcount\cols
{\catcode`,=\active\catcode`|=\active
	\gdef\Young(#1){\hbox{$\vcenter
			{\mathcode`,="8000\mathcode`|="8000
				\def,{\global\advance\cols by 1 &}%
				\def|{\cr
					\multispan{\the\cols}\hrulefill\cr
					&\global\cols=2 }%
				\offinterlineskip\everycr{}\tabskip=2pt
				\dimen0=\ht\strutbox \advance\dimen0 by \dp\strutbox
				\halign
				{\vrule height \ht\strutbox depth \dp\strutbox##
					&&\hbox to \dimen0{\hss$##$\hss}\vrule\cr
					\noalign{\hrule}&\global\cols=2 #1\crcr
					\multispan{\the\cols}\hrulefill\cr%
				}
			}$}}
}

\title[Surface Algebras I: Dessins D'enfants, Surface Algebras, and Dessin Orders]
{Surface Algebras I: Dessins D'enfants, Surface Algebras, and Dessin Orders}

\author[A.~Schreiber]{Amelie~Schreiber}
\email{amelie.schreiber.math@gmail.com}


\date{\today}

\begin{document}

	\begin{abstract}
	In this paper, a construction of an infinite dimensional associative algebra, which will be called a \emph{Surface Algebra}, is associated in a "canonical" way to a dessin d'enfant, or more generally, a cellularly embedded graph in a Riemann surface. Once the surface algebras are constructed we will see a construction of what we call here the associated \emph{Dessin Order} or more generally the \emph{Surface Order}. The surface orders can be constructed in such a way that they are the completion of the path algebra of a quiver with relations which gives the surface algebra of the dessin. This provides a way of associating to every algebraic curve $X$, with function field $k(X)$ (defined over an arbitrary field $k$) the representation theory of its Surface Algebra and the lattices over Surface Orders, which are defined as pullbacks of certain matrix algebras over commutative $k$-algebras. We will then be able to prove that the center and the (noncommutative) normalization of the surface orders are invariant under the action of the absolute Galois group $\cG(\overline{\QQ}/\QQ)$. We will see that the surface algebras and surface orders are closely related to the fundamental group(oid) of the Riemann surfaces and the associated monodromy group. A description of the projective resolutions of the simple modules over the surface order is given and it will be shown that one can completely recover the dessin with the projective resolutions of the simple modules alone. In particular, the projective resolutions of the simple modules encode all combinatorial and topological data of the monodromy group (or cartographic group) of a dessin. Finally, as a corollary we are able to say that classifying dessins in an orbit of $\cG(\overline{\QQ}/\QQ)$ is equivalent to classifying dessin orders with a given normalization. We end with some further examples of surface algebras and surface orders related to the classical and geometric version of the Langlands Program, and we give a sketch of some results to appear in an upcoming paper on Artin L-functions. 
	\end{abstract}

\maketitle

\textcolor{darkred}{\hrule}

\tableofcontents

\textcolor{darkred}{\hrule}
\medskip
\section{Introduction}

This paper is the beginning of an attempt to bring together several quite deep areas of mathematics, without being anything like an "expert" in more than one or two of them, if any. In particular, it aims to begin exposing techniques from quiver representation, and more generally, infinite dimensional associative algebras, incorporating ideas from modular representation theory of finite groups and group algebras and closely related Brauer graph algebras and gentle algebras. There are some applications of the theory of "\emph{lattices over orders}" used, and some use of basic covering theory of Riemann surfaces, Galois groups acting on extensions of number fields, and related combinatorics. There are indications of relations to loop groups, loop algebras and the Geometric Langlands Program, although this is a task we will take on in the followup papers \cite{AS2} and \cite{AS3}. The goal is not simply to present new results, but provide a somewhat new perspective on many old ideas and to begin building some kind of meaningful amalgamation of them all. There are many techniques which could be considered quite obscure, often inaccessible to most, and lacking clear and welcoming expositions. However, there are intuitive ways of approaching these tools used by the experts, and some explanation of these ideas are provided here. One of particular importance is the so-called functorial filtration technique, which was inspired by the Gel'fand-Ponomarev paper \cite{GP}, and first explained in obscure notes by P. Gabriel in \cite{G}, which are nowhere to be found aside from the "\emph{Benson archive}", and which were intended only for a small group of individuals meeting to discuss \cite{GP} and related representation theory in the early 70s. One resource which is also useful for this idea, and which will prove to be useful for ideas presented here is \cite{Hi}. 

In the following article we will define an infinite dimensional associative algebra which will be associated to a \emph{dessin d'enfant} (really to any cellularly embedded graph in a Riemann surface) in a very natural way which could be justifiably called a canonical construction. These algebras, which will be called \emph{surface algebras}, can be defined in terms of combinatorial objects determined by a dessin d'enfant, a cellularly embedded graph in a Riemann surface, or a Fuchsian or triangle group, called a \emph{quiver with relations}. The surface algebras are an infinite dimensional generalization of several classes of finite dimensional algebras which have been heavily studied since at least as far back as the classic paper by Gel'fand and Ponomarev on the \emph{Indecomposable Representations of the Lorentz Group} \cite{GP} and Gabriel's subsequent \emph{"functorial interpretation"}. In particular, these generalize Brauer graph algebras, group algebras of finite groups over algebraically closed fields of arbitrary characteristic, the so-called finite dimensional "gentle algebras", "string algebras", and "special biserial algebras". The completions of the surface algebras to thier surface orders, and other more general related surface orders are generalizations of "$R$-orders", used in the study of modular representation theory of group algebras $RG$ of finite groups $G$. So, once the surface algebras are defined, we define their $\fm$-adic completions with respect to the \emph{"arrow ideal"} $\fm$, and more generally we define certain formal matrix rings called \emph{surface orders}. Once these are defined we compute the center and the \emph{"noncommutative normalization"} and show that for the case of dessins d'enfants, these are invariant under the action of the absolute Galois group of the rationals $\cG(\overline{\QQ}/\QQ)$. For the case of arbitrary cellularly embedded graphs in Riemann surfaces, we identify the Galois groups of coverings of Riemann surfaces with quotients of the corresponding fundamental group, and we define an action of this group on the surface orders. This provides us with a way of understanding $n$-dimensional representations of the absolute Galois group $\cG(\overline{K}/K)$ of a number field $K$, and how they correspond to representations of $\GL_n(\bA_K)$, over the ring of adeles of $K$. More generally, it provides a way of defining an action of the Galois group of a field extension on certain matrix rings, over any discrete valuation ring or Dedekind domain. In particular, this will also work in the case of function fields and the GLP. This provides a way of defining characters of the cyclic subgroups of Galois groups (or indeed any automorphism group of a finite extension of number fields $K/F$, not necessarily Galois). We can then apply Artin's Theorem along with some invariant theory of $\GL_n$ to compute all characters of the Galois group and we can show that the equations defining characters can be written as $\ZZ$-linear combinations of characters of the cyclic subgroups. Moreover, this allows us to identify the characters defining Artin L-functions with semi-group rings which are coordinate rings of toric varieties, generated by polynomial invariants of a connected reductive algebraic group acting on the representations of the surface algebra and the surface order. 

\section{Historical Background and Motivations}

This section is not essential to the rest of the paper and may be skipped by anyone not interested in the historical motivations or the background on the theoretical tools used.

\subsection{Dessins D'Enfants and Cellularly Embedded Graphs}
Dessins d'enfants are, in the simplest terms, bipartite graphs embedded in a compact Riemann surface (without boundary). They are combinatorial tools used in the study of Inverse Galois Theory, orbits of the absolute Galois group of a number field, and \emph{"Belyi functions"}, which are coverings of the sphere $\PP_{\CC}^1$ ramified at most at three points, which can be assumed to be $\{0, 1, \infty\}$ up to a M\"{o}bius transformation. One way of constructing dessins is via a combinatorial object called a "\emph{$3$-constellation}" $C = [\sigma, \alpha, \phi]$, which is a triple of permutations such that,
\begin{enumerate}
\item The group $G = \la \sigma, \alpha, \phi \ra$ acts \textit{transitively} on $[n]$. 
\item The product $\sigma \alpha \phi = \id$, is the identity. 
\end{enumerate}
If we restrict to the case when $\alpha$ is a fixed point free involution, we obtain part of the standard definition of a Brauer graph (see the next section for details). In the case where we do not assume this restriction, this gives a \emph{"hypermap"} which corresponds uniquely to a bipartite graph embedded in a compact Riemann surface and to a dessin d'enfant. It also give a presentation of a triangle group 
\[ \Delta(p,q,r) := \la \gamma_0, \gamma_1, \gamma_{\infty} | \gamma_0^p = \gamma_1^q = \gamma_{\infty}^r = \gamma_0 \gamma_1 \gamma_{\infty} = 1 \  \ra,\]
(see \cite{JW}). These are special \emph{"Fuchsian groups"}. More generally, we may associate to a cellularly embedded graph a cocompact Fuchsian groups with generators $\{\gamma_j, \alpha_j, \beta_I\}$ relations of the form
\[ \gamma_1^{a_1} = \cdots = \gamma_r^{a_r} = \prod_{i=1}^g [\alpha_i, \beta_i] \prod_{j=1}^r \gamma_j = 1 \]
From such groups we may construct arbitrary Brauer graph algebras in a more or less canonical way, as a quotient of the corresponding surface algebra.

\subsection{Brauer Graph Algebras}
Brauer graph algebras are a class of algebras construct from a so-called \emph{"Brauer graph"}, which is defined as follows:

\begin{defn}
Fix any finite connected graph possibly with loops. Define a "cyclic ordering" of the edges around each vertex, then define a multiplicity $m_v$ of each vertex $v$, given by a positive integer $m_v \in \ZZ_{\geq 1}$. 
\end{defn}

The Brauer graph algebras are all finite dimensional associative algebras which can be obtained from a \emph{"quiver with relations"} (see \cite{B1}, \cite{E}, \cite{ASS}).  The quiver of the Brauer graph algebra associated to the Brauer graph can be realized as the same quiver as that obtained in the construction of Surface Algebras in \ref{Surface Algebras}. The relations on the quiver are quite cumbersome to define and are traditionally in terms of the simple and projective modules of the algebra, and extensions between them. For details the reader is referred to \cite{B1}, and \cite{E}. The motivation for the construction of such algebras comes from the modular representation theory of finite groups. In particular, for an algebraically closed field $K$ of arbitrary characteristic, the group algebra $KG$ is not in general semisimple. If the characteristic $\mathbf{char}(K)=p>0$ divides the order of the group $|G|$, then it decomposes into indecomposable "\emph{blocks}"
\[ KG = B_1 \oplus B_2 \oplus \cdots \oplus B_r.\]
It is well known such block are always Brauer graph algebras. There is extensive research on this subject by Karin Erdmann (see for example \cite{E, E1, E2, E3, E4}). A more complete list of references can be found in \cite{B1}. 

Using the methods developed in \cite{GR} and the Morita equivalence of a ring $R$ with the matrix ring $M_n(R)$ we are able to glue matrix rings together to obtain representations of automorphism groups of field extensions (not necessarily Galois). This can be interpreted geometrically via a pullback of vector bundles/sheaves with flat connection as in \cite{AS2}. We use this and the invariant theory developed in \cite{CCKW}\footnote{This paper uses methods developed in \cite{C}, \cite{CW}, and \cite{CC} to describe (semi)invariant rational and polynomial functions under the action of various reductive groups $\prod \GL(d_i)$ for a dimension vector $d = (d_1, d_2, ..., d_n)$ using the theory of Schur functors.}

\subsection{Gel'fand and Ponomarev, Indecomposable Modules of the Lorentz Group, and Special Biserial Algebras}
In \cite{GP}, Gel'fand and Ponomarev studied the indecomposable representations of the Lorentz group. In so doing they define a class of modules which they call \emph{"Harish-Chandra modules}. The study of such modules can be reduced to the study of indecomposable representations of the algebra $k[[x, y]]/(xy)$, as explained by Gabriel in \cite{G} in his functorial interpretation of the methods used in \cite{GP}. The paper \cite{GP} inspired extensive research in the years to come, which can still be seen in current research. In particular, the so-called \emph{"functorial filtration"} method, was then used along with aspects of covering theory in some cases in \cite{R1}, \cite{BR}, \cite{CB1, CB2, CB3, CB4}, \cite{He1, He2, He3, He4}, and potentially a countably infinite set of others. This method at its core is applying \emph{induction} and \emph{restriction} of representations in more or less the usual sense, along with some (directed) combinatorial algebraic topology and geometric group theory related to actions of groups on trees (see for example \cite{Gr}). These ideas seem to have already been more or less completely understood already by Gabriel and Ringel in \cite{G} and \cite{R1} in the early 70s following \cite{GP}, but seem to have been a technique not widely known for quite some time beyond a small group of people who applied the techniques regularly to prove many complicated results about the module categories of associative algebras. The class of algebras most heavily studied using these methods are the \emph{"special biserial algebras"}, which contain the class of Brauer graph algebras, as well as the so-called \emph{"gentle algebras"} and \emph{"string algebras"}. These are generalizations of the Gel'fand Ponomarev algebras, and Ringel's \emph{"dihedral agebras"} in \cite{R1}. 

It can be shown with very little work that every special biserial algebra is a quotient of a surface algebras as we have defined it here, and the universal cover of \emph{all} surface algebras is a bi-colored, directed Cayley graph of the free group on two generators (i.e. the underlying graph forgetting the orientation of the arrows is a four-regular tree). So, in some sense, the free group on two generators yields every group algebra of a finite group over algebraically closed fields of arbitrary characteristic via a quotient by a group action, and possibly additional quotients thereafter yielding admissible ideals. So, in a way, the surface algebras are a unifying class of algebras.

\subsection{Cluster Algebras and Surfaces}
Algebras defined using some combinatorial construction associated to a triangulation of a compact surface \emph{with boundary} have shown up in various places. One of the most significant is in the theory of cluster algebras of surface type (see for example \cite{FST}) where a cluster algebra structure was given to triangulations of compact Riemann surface without punctures. In \cite{ABCP} a quiver is associated to a triangulation of a compact Riemann surface, motivated by the association of a cluster algebra to a triangulated surface in \cite{FST}, and earlier similar constructions, and the fact that a quiver can always be associated to any skew-symmetric matrix (which is the starting point of the theory of cluster algebras). For more information on cluster algebras we refer to \cite{FZ1, FZ2, BFZ, FZ4}\footnote{The original motivation for the study of cluster algebras was to understand "canonical bases" and "total positivity". These ideas are prevalent in the work of Kashiwara (crystal bases and quantum groups), and Lusztig, who has produced a mass amount of work on p-adic representation theory of reductive algebraic groups (among other things), and various investigations into the Langlands philosophy. This work, especially the theory of Kazhdan-Lusztig polynomials will be important in understanding actions of Galois groups realized as subgroups of affine Coxeter groups acting on root systems corresponding to lattices over surface orders.} and \cite{DWZ1, DWZ2}\footnote{This work provides a connection to the theory of cluster algebras via quiver representations.}.

In nearly every case associating a quiver with \emph{"gentle relations"} to a surface, the gentle algebra is constructed from a \emph{triangulation} of a compact surface. In the case of Brauer graph algebras, it has recenty been noticed that such algebras can be related to a surface, for example in \cite{ES1, ES2, ES3, ES4}, \cite{GSS}, \cite{GSST}, \cite{MSc}, \cite{Sch}. The full implications and appropriate generalizations have yet to be understood. Further, all of these algebras are finite dimensional. Although it seems natural to extend the definition of these "surface algebras" to arbitrary cellularly embedded graphs in compact surfaces, to surfaces \emph{with or without} boundary and marked points, and to infinite dimensional algebras, little work in this direction has been done. Some ideas in this direction are hinted at, but a full generalization of such algebras to the case contained in this paper has yet to be constructed. Some examples which have only shown up very recently are the infinite dimensional string algebras in \cite{CB4}, the "surface algebras" recently constructed in \cite{ES1, ES2}, which are finite dimensional quotients of the surface algebras defined in this paper, and \cite{BH}. The construction of this paper yields a generalization of the constructions given from Brauer graph algebras, group graded algebras, special biserial algebras (gentle algebras especially), and cluster algebras of surface type. Moreover, all of these can be seen as quotients of the surface algebras defined in this paper, so many of the results from these various areas should be expected to generalize into a universal theory dealing with them all simultaneously.

\subsection{Covering Theory, Noncommutative Geometry, and Inverse Galois Theory}
In \cite{Gr}, \cite{GSS}, \cite{Sch}, \cite{He1, He2, He3, He4} methods of covering theory of associative algebras are applied. The paper \cite{GSS} give a construction of a tower of algebras similar in spirit to towers of groups that one might encounter in inverse Galois problems. Since group algebras (or their indecomposable blocks) are all Brauer graph algebras, this clearly has applications to Inverse Galois Theory. One would hope that extending and applying these methods to surface algebras and surface orders as defined in this paper would provide a "universal covering theory" for all special biserial algebras, as all such algebras are quotients of surface algebras. Moreover, the interpretation given here of the action of Galois groups yields a direct connection number theory and potential applications to Inverse Galois Theory. It is likely some form of "directed algebraic topology" will be the appropriate context for such a theory (see for example \cite{Hi}, \cite{Gra}). 

Moreover, this implies thinking of such topological constructions for quivers as a discrete (or algebraic) version of noncommutative (differential) topology. The "arithmetic noncommutative geometry" developed in \cite{Ma} should be the proper context for this theory applied to the \emph{surface orders} at infinite places of a number field (surface orders being an obviously arithmetic component of the theory developed in this paper). An overall goal should be to attempt to have a triad of theories all tied together by the theory of surface algebras and surface orders. The (directed/noncommutative) combinatorial algebraic topology used in \cite{Hi} and \cite{Gra} should be related to the "quiver component" of the theory, i.e. surface algebras; the arithmetic noncommutative geometry should be the surface orders over $p$-adic fields component of the theory; and the differential noncommutative geometry should be the proper tool for infinite places and "irrational limiting behavior". This should all tie together to form a complete local and global (adelic), and function field version of the theory. The work of Lebryun in \cite{L} will likely prove useful as well for a more algebro-geometric interpretation and for the study of moduli stacks.

From this noncommutative geometry point of view, we find connections to \cite{Ma} and \cite{Co}. The viewpoint of treating surface algebras and surface orders as discrete approximations of gluings of noncommutative tori (see \cite{Ma} and the material on noncommutative elliptic curves), has a discrete counterpart in the theory of surface algebras. In particular, realizing paths on a surface as modules in a module category (or lattices over orders, which by \cite{RR} is compatible with the quiver version), one has an interpretation of Hom and Ext spaces (morphisms and extensions in the module category or derived category), as intersections of paths on the surface. This provides some notions of "closeness" and "distance" in representation theoretic/discrete terms. Further, applying some of the ideas in \cite{G}, using the Cantor filtration of the module category of a surface algebra (or perhaps a generalization to the derived category), one may put a linear order on indecomposable modules, and the Hom and Ext spaces are then easily computed by lifting modules and applying the functorial filtration method. More remarks on this "modules as paths and/or geodesics" will be ellaborated on further below.

\subsection{$R$-Orders and Formal Matrix Rings}

One possible reason for not considering the class of algebras presented here is that they are infinite dimensional, and so the representation theory is thought to be much more difficult. Due to work of Crawley-Boevey \cite{CB4}, and results which can be deduced from \cite{R1}, \cite{GP}, and \cite{BR}, one can classify all indecomposable modules for the surface algebras as we have defined them. Moreover, the technology of "\emph{$R$-order}", which are a class of formal matrix rings (see \cite{KT}), allows one to realize the completions of the surface algebras with respect to the (maximal) arrow ideal. This allows one to use the methods of \emph{"lattices over orders"} to study the surface algebras. This technology also arose in the study of modular representations of finite groups and group algebras. We employ some of both perspectives, as both are useful and illuminating in their own way. For material on $R$-orders the reader is referred to \cite{CR1, CR2}, \cite{RR}, \cite{Ro1, Ro2}, \cite{K}, \cite{KR}.

\subsection{Nonabelian Class Field Theory}

An \emph{"adequate"} formulation of nonabelian class field theory, according to some of the literature (see for example \cite{Fr1, Fr2}, \cite{Ge}), seems to be very elusive, or at the very least, not so intuitive and difficult to state without an extensive background in abelian class field theory, reductive algebraic groups, algebraic geometry, and complex analysis. Some of the constructions contained in this paper and subsequent papers building off of the foundations given here will hopefully provide a new perspective on nonabelian class field theory which do not require quite as much background material as preparation, and as many of the constructions are very combinatorial, many of the problems and techniques can be phrased in very elementary terms, allowing more extensive study of the Langlands Program and its geometric generalization.

\subsection{Reductive Algebraic Groups and the Classical Langlands Program}

The technology of $R$-orders which is developed in this paper, inspired heavily by Kauer, provides a very concrete and clear interpretation of $n$-dimensional representations of a Galois group and how they correspond to the so-called \emph{"automorphic representations"} of reductive algebras groups. In particular, we are able to understand the action of the general automorphism groups of arbitrary extensions of number fields $F/K$, not necessarily Galois. The realization of such representations is given by an action on a pullback of Lie algebras corresponding to ramified primes. In particular, the constructions gives a way of understanding representations for \emph{ramified and unramified} primes simultaneously, \emph{without} having to exclude \emph{any} information regarding the ramified primes. This plays an important part in applications to L-functions (especially Artin L-functions). 

The representation are given by the \emph{"surface orders"} coming from completions of the \emph{"surface algebra"} defined over complete discrete valuation rings, Dedekind domains, and local and global fields. In particular, these constructions provide a way of treating the local and global aspects simultaneously as parts of a single unifying object. Let $K/\QQ$ be a finite extension. Using the construction given here, one can treat representations of the Galois group $\mathcal{G}(K/\QQ)$ of a field extension over the global field $K$, the adeles $\mathbf{A}_K$, local fields $K_{\nu}$ corresponding to \emph{finite or infinite} places, the local ring of integers $\mathcal{O}_{\nu}$, global ring of integers $\mathcal{O}_K$, and actions on the finite fields given by the corresponding extensions of $\ZZ/p\ZZ$ (or $\ZZ_p/p\ZZ_p$).

\subsection{Loop Groups, Loop Algebras, and the Geometric Langlands Program}

The ideas developed for the arithmetic Langlands Program generalize immediately to the geometric generalization of the Langlands Program (see for example \cite{Fr1, Fr2}). It gives a way of understanding representations as, sections of vector bundles over smooth projective algebraic curves defined over an arbitrary field. For the complex case, it gives a way of understanding all of this in the framework of representations of loop groups and loop algebras, and pullbacks of them, and a corresponding Kac-Moody Lie algebra given by a central extension of the pullback. The preprint \cite{M} explains how one may interpret this situation in terms of affine Schubert varieties, which is more or less an exposition on work of Lusztig. Lusztig's work on $p$-adic representations theory is simply too extensive to delve into, but the construction mentioned in \cite{M} is clearly applicable as shown in \cite{AS2}. This also means that the theory of standard monomials may be applied, and in future work this method will be applied to study semi-invariants and moduli stacks of representations of surface algebras.

\subsection{Artin L-functions}

From the constructions given within this paper, we may deduce some significant 
results about Artin L-functions and the Langlands correspondence. In particular, the characteristic polynomials defining Artin L-functions are polynomial invariant under base change. So, in a later paper we provide a way of computing the generators of the rings of polynomial and rational invariants. We are able to show that the invariants (and more generally semi-invariants) define a semi-group ring which are coordinate rings of toric varieties. One can then hope to determine whether general Artin L-functions $L(\rho, \chi_i)$ (not necessarily corresponding to automorphic representations), coming from the irreducible characters $\chi_1, ..., \chi_r$ of a Galois group $\cG$ are holomorphic. The realization of the Artin L-functions coming from the semi-simple automorphic representations are shown to be normal. Moreover, it can be shown that any Artin L-function coming from so-called semi-stable "regular modules" are holomorphic. An in depth study requires constructions from Geometric Invariant Theory, constructions given in \cite{LP}, \cite{Do1, Do2, Do3, Do4}, and \cite{SV}, \cite{D}, \cite{DW}, \cite{C}. One can frame this in terms of combinatorial commutative algebra as in \S 7, \S 10 of \cite{MS}, which provides a very approachable and computational way of understanding the problem, and does not require extensive background in so many different areas. Example 10.13 pg. 197 is of particular significance. Results and ideas from \cite{Ci1}, \cite{CN1, CN2, CN3} and \cite{N1, N2, N3}, are then understood more clearly via the representation theory and invariant theory. Applying this along with Artin's Theorem (see \cite{S1} \S 12.5 and \S 17.2) applied to the cyclic subgroups of the Galois group corresponding to the \emph{noncommutative normalization} of the surface order (or more generally, to the automorphism group of a non-Galois extension of arbitrary number fields) we may obtain the following statements:

\textbf{Theorem \ref{Main Theorem}}
\begin{enumerate} 
\item \emph{The rings of polynomial semi-invariants (under an action of special linear groups), and therefore of the polynomial invariants under arbitrary base change are all semi-group rings and are the coordinate rings of affine toric varieties. }

\item \emph{Moreover, the parametrizing varieties of representations, i.e. the "representation varieties," of fixed dimension for a given surface algebra are normal, Cohen-Macaulay, and have rational singularities. }

\item \emph{By results of \cite{C}, \cite{CC}, and \cite{CCKW}, one may deduce that the moduli spaces of the semi-stable "regular modules" are in fact isomorphic to projective lines, and thus all Artin L-functions corresponding to such representations are holomorphic.}\footnote{An irreducible component of a module/representation variety is called \textbf{regular} if the generic module is a direct sum of so-called \emph{band modules}.}
\end{enumerate}

A complete set of generators and relations of semi-invariant polynomials and rational invariants can be given explicitly. Further, the isomorphism classes of representations of a surface algebra are given by orbits under certain connected reductive groups. The parametrizing varieties of these representations can be well understood. In particular, using these results, we can then understand the moduli spaces of representations of surface algebras completely. Further, all of these results can be stated in a characteristic free way over an algebraically closed field, and many can even be stated over $\ZZ$. So, an understanding of automorphic representations, their parametrizing varieties and moduli spaces, the invariant polynomials, and rational invariants can all be understood completely. How the representations of the automorphism group of a finite extension of number fields can be understood in these terms is then very concrete, and from this a rather complete picture of Artin L-functions in terms of polynomial and rational invariants can be obtained. In particular, the L-functions associated to an isomorphism class of representations can be computed explicitely and classified in terms of rings of semi-invariants under base change. For more information on how one may understand this approach see \cite{M}, \cite{C}, \cite{CC}, and \cite{CCKW}. 

From all of this we may obtain the following

\textbf{Corollary} \ref{Main Corollary Artin L-functions}
\emph{The Artin L-functions for any finite field extension $K/F$ of number fields can be realized as 
\[ L(s, \chi) = m_1^{a_1}m_2^{a_2} \cdots m_r^{a_r} \]
a product of monomials corresponding to the local orders of a surface order for any character $\chi$ of the group $G(E/F)$. Moreover, we have that
\[ \chi = a_1\chi_1 + a_2\chi_2 + \cdots + a_r\chi_r \]
where the $\chi_i$ are characters of the cyclic subgroups of $G$ corresponding to the local orders.}

\subsection{Geodesics on Riemann Surfaces and Maryam Mirzakhani's Work}

Using constructions reminiscent of \cite{B}, \cite{BM}, \cite{BS}, \cite{BT}, one may show a certain kind of correspondence between geodesic flows on Teichm\"{u}ller space, and indecomposable modules of surface algebras. The representation theory of surface algebras is essentially known by results of Crawley-Boevey, Ringel, Gabriel, and Gel'fand-Ponomarev. So methods from representation theory of quivers can be applied to study geodesics in moduli spaces of curves as in \cite{EM}, \cite{EMR}, and \cite{MZ}. In particular, using some ideas from \cite{G}, some number theoretic results may be obtained. This is provided in an upcoming paper along with a detailed description of Gabriel's methods in \cite{G} (as this paper is not readily available but contained some important intuition for this application), along with some new examples interpreted in terms of combinatorial commutative algebra, and lattices over orders. As mentioned above, the work of \cite{G}, \cite{Ma}, and \cite{Co} will become important in developing these aspects. 

One should note at this point, a good indication that the "noncommutative geometry" point of view is an important line of investigation is given by Lemma 2.4, pg. 28 and the material following the Lemma in \cite{Ma}. This single Lemma shows one may think of the "shift operator" for the local surface orders defined here, as a version of the shift operator in the Lemma. Moreover, this provides immediate connections to geodesic flows explained in \cite{Ma}, which of course implies connections to Maryam Mirzakhani's work. Using this conceptual framework, and relating the Galois group (or fundamental group) to the action of shift operators, then using the Cantor filtration on modules over surface algebras may yield results on what kinds of geodesics give periodic or dense orbits. This is one important aspect of the theory of "billiards". 

One may then of course apply a straightforward M\"{o}bius transform taking the shifted half plane $\mathbb{H}_{\zeta} = \{z \in \CC: \ Re(z) < 1/2\}$ to the upper half plane $\mathbb{H}_+$. Using the combined work of \cite{Ma} (building on the work \cite{Co}, where Riemann $\zeta$-functions are studied), and \cite{EM}, \cite{EMR}, and \cite{MZ}, and perhaps applications of Selberg $\zeta$-functions and trace formulas (along with the interpetation of path in a surface as modules via generalizing \cite{BM, BS, BT}), one might obtain some very interesting results $\zeta$-functions and special values of L-functions. Further, our treatment here of Artin L-functions in terms of determinantal rational invariants seems to be justified by the expression of the Selberg $\zeta$-function of a modular curve as a Fredholm determinant of the Perron-Frobenius operator as explained in \cite{Ma} pg. 29. 

\subsection{Quantum Computing, Topological Surface Codes, and Quantum Gravity}

Seeing the connections to noncommutative geometry in \cite{Ma}, \cite{Co}, and \cite{L}, it should be of no surprise to the reader that applications to physics might arise. In particular, we should mention some of the more obvious and immediate applications to "topological surface codes" and "quantum computing", and related theories of quantum gravity. Building the first working prototypes of quantum computers has now been accomplished to some degree, so it seems developing a theory that provides a conceptual framework and some intuition to these problems would be important. In \cite{BCCT}, \cite{Br1}, \cite{BCKTV}, \cite{BT}, it is clear that surface algebras and surface orders may be applied to such problems. In particular, one may interpret actions of fundamental groups or Galois groups\footnote{Or generalizations of them to the motivic setting, and closely related groups such as the conjectural Langlands groups. This will require some applications of Section \S \ref{Resolutions} to Beilinson's conjectures via the Serre-Swan Theorem and a mixture of topological and algebraic K-theory.} as the "\emph{Pauli operators}. In particular from pg. 49 of \cite{Br1} we have, 

\emph{"Each stabilizer $X$-check corresponds to a vertex of the tessellation, acting on all edges incident
to the vertex. In a $\{r,s\}$-tessellation the weight of any $X$-check is therefore $s$. Similarly, each
stabilizer $Z$-check corresponds to a face of the tessellation, acting on all incident edges. The weight
of every $Z$-check is therefore $r$. Note that every edge, regardless of $r$ and $s$, is incident to two faces
and two vertices, so that every qubit is acted upon by 4 stabilizer checks (qubit degree is 4)."}

in reference to "hyperbolic surface codes", the vertices and edges in the quote being those of a cellularly embedded graph in the hyperbolic surface as used in the construction of surface algebras and surface orders in this paper. 

Expanding on this correspondence one can investigate applications of Maryam Mirzakhani's work on asymptotic enumeration formulas as a means of studying error detection and correction in quantum surface codes. In particular, the interpretation of a state of the qubits associated to the cellularly embedded graph in the surface as a particular module over the surface algebra (and as paths on the surface), one has again a notion of disctance via the Cantor filtration provided by Gabriel in \cite{G}, as well as a linear order of indecomposable modules which induces a linear order on states of the qubits. This would allow one to understand the possible histories of states as well as possible future states of a current state of the qubits. Moreover, this provides a very combinatorial way of computing such things which are easily implemented in a computer program.\footnote{For example, the computer algebra packages, Magma, GAP, and Sage all have means of handling such computations in terms of quiver representations.} 

Currently, there are some theories of quantum gravity which very closely resemble that of a quantum surface code (see for example \cite{Wi1, Wi2, Wi3}, \cite{RW}). It is not such a stretch to transfer the ideas for surface codes just mentioned to the realm of quantum gravity. In particular, in a Github code repository \cite{AS4} we show how to implement the surface algebras on a quantum computer using \emph{Google Cirq}. The physical implications of the gluing of matrix rings seems to correspond to some spectral compatability constraints on the Hamiltonian of what we call a \emph{Hybrid Qudit Surface Code}. A \emph{Hybrid Qudit Surface Code} is simply a surface code with qudits of mixed dimensions which corresponds to a dessin d'enfant. Each vertex of the dessin corresponds to a qudit with dimension equal to the number of half edges at the vertex. The spectral constraints on the Hamiltonian imply an implementation of the (truncated) Hilbert-Polya conjecture for arbitrary Artin L-functions.


\section{Dessins D'Enfants and Combinatorial Embeddings}\label{Dessins}

In this section we will set our notation and terminology for cellular embeddings of graphs in Riemann surfaces and dessins. We will use some basic combinatorial topology, and although more attention is given to dessins d'enfants, this setup works perfectly well for arbitrary graphs cellularly embedded in Riemann surfaces in general, and very slight modifications can be made to handle the case of noncompact surfaces and surfaces with boundary as well. It is also important to note, that the presentation of the monodromy group of a branched covering
\[ \beta: X \to \Sigma \]
of the Riemann sphere, can be written as a $2$-generator transitive subgroup of $S_n$, where $n$ is the degree of the map $\beta$. The presentation is $G = \la g_0, g_1 \ra$, where $g_0$ gives the permutations describing the branching patterns above $0$, and $g_1$ gives the branching patterns above $1$ when $\beta$ is a Belyi function. Then the permutation $g_{\infty} = (g_0g_1)^{-1}$ describes the branching patterns above $\infty$. In what follows, the setup is for arbitrary cellularly embedded graphs, but the modification to dessins is straightforward and is given by not restricting what we call $"\alpha"$ below to fixed-point free involutions. In other words, we allow arbitrary "constellations" in the language used in this paper. For an introduction to this terminology see \cite{LZ} and \cite{JW}.

\subsection{$3$-Constellations for Arbitrary Graphs on Riemann Surfaces}
Let $S_n$ be the symmetric group on $[n] = \{1,2,3...,n\}$. Permutations will act on the left, so if $\sigma \in S_n$, we will use the cycle notation. 
Let us define a $k$-\textbf{constellation} to be a sequence $C=[g_1, g_2, ..., g_k]$, $g_i \in S_n$, such that:
\begin{enumerate}
\item The group $G = \la g_1, g_2, ..., g_k \ra$ generated by the $g_i$ acts \textit{transitively} on $[n]$. 
\item The product $\prod_{i=1}^k g_i = \id$, is the identity. 
\end{enumerate}
The constellation $C$ has "\textbf{degree} $n$" in this case, and "\textbf{length} $k$". Our main interest will be in $3$-constellations $C = [\sigma, \alpha, \phi]$, which we will describe in detail momentarily. The group $G = \la g_1, g_2, ..., g_k \ra$ will be called the \textbf{cartographic group} or the \textbf{monodromy group} generated by $C$. 
\\

Let $\PP = \PP^1(\CC)$. Let $\Sigma$ be a compact Riemann surface. Suppose $\beta: \Sigma \to \PP$ is a \textbf{Belyi function}. It is often useful to visualize Belyi functions as combinatorial maps on $\Sigma$. This construction plays a big part theoretically since it gives a way of using algebraic and combinatorial methods to study Belyi functions, and it also gives very concrete examples which are useful for developing intuition. Such combinatorial maps uniquely determine $\Sigma$ as a Riemann surface or algebraic curve, and they uniquely determine the Belyi function $\beta$. In fact, the Riemann surface is determined over an algebraic number field if and only if its complex structure is obtained from such a combinatorial map. In particular, we have

The following important Theorem can be found in \cite{JW}. 
\begin{theorem}\label{Belyi's Thm}
(\textbf{Belyi's Theorem}): Let $X$ be a compact Riemann surface, i.e. a smooth projective algebraic curve in $\PP_{\CC}^N$ for some $N$. Then $X$ can be defined over $\overline{\QQ}$ if and only if there exists a nonconstant meromorphic function $\beta: X \to \PP_{\CC}^1$ ramified over at most three points. 
\end{theorem}

\begin{defn}
Any meromorphic $f$ fitting the criteria of Theorem \ref{Belyi's Thm} is called a Belyi function. For such $f$, place a black vertex $\bullet$ at $1 \in \PP_{\CC}^1$, and a white vertex $\circ$ at $0 \in \PP_{\CC}^1$, and an edge on the real interval $[0,1]$. This gives a bipartite graph $\Gamma_{\PP}$ embedded in the sphere $\PP_{\CC}^1$. Define a \textbf{dessin d'enfant} to be $f^{-1}(\Gamma_{\PP}):=\Gamma \subset X$. It will have the structure of a bipartite graph, cellularly embedded $\Gamma \hookrightarrow X$, in the compact Riemann surface$X$. One may equivalently define a dessin to be an \emph{algebraic bipartite map} (see \cite{JW}), which is given by an arbitrary $3$-constellation $C = [\sigma, \alpha, \phi]$. 
\end{defn}

There is a correspondence between $3$-constellations $C = [\sigma, \alpha, \phi]$ such that $\alpha$ is a fixed point free involution, and graphs which are cellularly embedded in a compact Riemann surface without boundary. In particular, such constellations give a CW-complex structure on the surface $X$. Dessins are then the special cases when $C$ corresponds to a bipartite graph embedded in $X$.

\subsection{The Clockwise Cyclic Vertex Order Construction}
There are many equivalent ways of defining a graph on a Riemann surface. One of the simplest and probably the most combinatorial ways is by constellations. There are at least two ways of viewing this construction. We present two here, which are in some sense dual to one another. Intuitively, we follow the recipe:

\begin{enumerate}
\item First choose some positive integer $r \in \NN$ to be the number of \emph{vertices} 
of the graph, say $\Gamma_0 = \{x_1, x_2, ..., x_r\}$. 
\item Then, to each vertex $x_i$, we choose some number $k_i$, of \emph{"half edges"} to attach to it, with the rule that once we have chosen $k_i$ for each $x_i$, the sum $\sum_{i=1}^r 
k_i = 2n$, must be some positive even integer. 
\item We then choose a \emph{clockwise cyclic ordering} of the \textit{"half-edges"} around 
each vertex $x_i$, i.e. some cyclic permutation $\sigma_i$ of $[k_i] = \{1, 2, ..., k_i\}$ for each $x_i$. The 
cyclic permutations $\sigma_i$ must all be disjoint from one another, and together they form a permutation of $[2n]$.
\item Once such a \emph{cyclic ordering} is chosen, we then define a \emph{gluing} of all of the \emph{"half edges"}. In particular, we choose some fixed-point free involution on the collection of all half edges, which is a permutation in $S_{2n}$ given by $|\Gamma_1|$ many $2$-cycles. This defines $\alpha$. 
\end{enumerate}

We then have the usual \textbf{Euler formula}, 
\[ |\phi|-|\alpha|+|\sigma| = F-E+V = \chi(\Sigma).\]
Said a slightly different way, we define a pair $[\sigma, \alpha]$, where $\sigma, \alpha \in S_{2n}$. The permutation
\[ \sigma = \sigma_1\sigma_2 \cdots \sigma_r \] is a collection of cyclic permutations, one $\sigma_i$ for each 
vertex $x_i$ of our graph $\Gamma = (\Gamma_0, \Gamma_1)$. So, perhaps in better notation, each 
\[ \sigma_x = (e(x)_1, e(x)_2, ...., e(x)_{k(x)}),\]
can be thought of as giving a cyclic ordering of the \textbf{half edges}.

Denote the half edges,
\[ \Gamma_1(x) = \{e(x)_1, e(x)_2, ..., e(x)_{k(x)}\},\] attached to each vertex $x \in \Gamma_0$ in our graph. The cycles $\sigma_x$ are all necessarily disjoint. We define how to glue pairs of 
\emph{half edges}, in order to get a connected graph $\Gamma$, via the permutation $\alpha$. 

The permutation $\alpha$ is of the form 
\begin{align*}
\alpha &= \alpha^1 \alpha^2 \cdots \alpha^t \\
		   &= (\alpha_1, \alpha_2)(\alpha_3, \alpha_4) \cdots (\alpha_{2n-1}, \alpha_{2n}) \\
		   &= \prod_{e \in \Gamma_1} \alpha(e)
\end{align*}
and each $(\alpha_i, \alpha_{i+1})$ tells us to glue the two corresponding \emph{half-edges}. Here we can also view \[ \alpha : \Gamma_1 \to \amalg_{x \in \Gamma_0} \Gamma_1(x) \] as a map from the edges $\Gamma_1$, to the \textit{half-edges} $\amalg_{x \in \Gamma_0} \Gamma_1(x)$. So $\alpha(e) = (e(x)_p, e(y)_q)$.

\begin{table} 
\centering 
\fbox{\begin{tabular}{l c c c c c} 
Notation & Meaning  \\ 
 \\
$\Gamma \hookrightarrow \Sigma$ & cellularly embedded graph \\ 
$\Sigma$ & a Riemann surface, generally closed \\ 
$\Gamma_0$ & the vertex set of a graph $\Gamma$ \\ 
$\Gamma_1$ & the edge set of a graph $\Gamma$ \\ 
$\Gamma_1(x)$ & the \textit{half-edges} around a vertex \\ 
$e(x)_i$ & a \textit{half-edge} in $\Gamma_1(x)$ attached to $x \in \Gamma_0$ \\ 
$\partial e = \{\partial_{\bullet} e, \partial^{\bullet} e\}$ & the vertices adjacent to $e \in \Gamma_1$ \\ 
$\alpha^k = (\alpha_i, \alpha_{i+1})$ & a $2$-cycle of $\alpha$ \\ 
$(\alpha_i, \alpha_{i+1}) = (e(x)_p, e(y)_q)$ & glued \textit{half-edges} $e(x)_p$ and $e(y)_q$ \\ 
$\alpha(e) = (e(x)_p, e(y)_q)$ & $\alpha$ as a map $\Gamma_1 \to \amalg_{x \in \Gamma_0}\Gamma_1(x)$ \\  
& \\
\end{tabular}}\\
\label{table1} 
\end{table}

\begin{ex}
Let us illustrate this by a simple example. As a permutation on the set of all \emph{half edges},
\[ \Gamma_1(x) \amalg \Gamma_1(y)  = \{e(x)_1, e(x)_2, e(x)_3, e(y)_1, e(y)_2, e(y)_3\},\] around two vertices $\Gamma_0 = \{x,y\}$ 
we may identify $\sigma, \alpha \in \mathbf{Perm}(\Gamma_1(x) \amalg \Gamma_1(y))$ with permutations in $S_6$. Namely, let us define the identification
\begin{align*}
\sigma &= \sigma_x \sigma_y \\
		   &= (e(x)_1, e(x)_2, e(x)_3)\cdot (e(y)_1, e(y)_2, e(y)_3) \leftrightarrow (1,2,3)(4,5,6) \in S_6
\end{align*}
and let 
\[ \alpha = (e(x)_1, e(y)_1)(e(x)_2, e(y)_2)(e(x)_3, e(y)_3) \leftrightarrow (1,4)(2,5)(3,6) \in S_6.
\] 
Then under this identification, the graph with
\begin{itemize}
\item two vertices $\Gamma_0 = \{x,y\} \leftrightarrow \{\sigma_x, \sigma_y\} = \{(1,2,3), (4,5,6)\}$, 
\item and six half edges \[ \Gamma_1(x) \amalg \Gamma_1(y)  \leftrightarrow \{1,2,3\} \amalg \{4,5,6\}.\] 
\end{itemize}
may be represented by the following picture to help visualize this, \\
\
\\
\begin{tikzcd}[column sep=.1in,row sep=.3in]
 &  & {\sigma = (1,2,3)(4,5,6)} &  & {\alpha = (1,4)(2,5)(3,6)} &  &  \\
 &  &  &  &  &  &  \\
 &  &  &  &  &  &  \\
 & \ \arrow[d, "1"', no head, bend right] \arrow[rrr, "{(1,4)}", no head, dashed, bend left=49] &  &  & \ &  & \ \\
 & {(1,2,3)} \arrow[rd, "3", no head, bend left] \arrow[ld, "2", no head, bend right] &  & {(2,5)} \arrow[rrru, no head, dashed, bend left=49] &  & {(4,5,6)} \arrow[lu, "6", no head, bend left] \arrow[d, "4", no head, bend left] \arrow[ru, "5", no head, bend right] &  \\
\ \arrow[rrru, no head, dashed, bend right=49] &  & \ \arrow[rrr, "{(3,6)}", no head, dashed, bend right=49] &  &  & \ & 
\end{tikzcd}
\\
\
\\
\end{ex}

\subsection{The Polygon Construction}
It is important at this point to make a few comments. Not every graph is planar, i.e. there may be no embedding on the sphere $S^2 = \PP^1$ without edge crossings. To see a second way this plays out with constellations, we now turn to the dual construction on faces. In the last section, the permutations $\phi = \alpha \sigma^{-1}$, defining the constellation $C = [\sigma, \alpha, \phi]$ were quite neglected in the construction. This is partially because they are not strictly needed since $\sigma \alpha \phi = \id \implies \alpha \sigma^{-1} = \phi$. 

The previous construction focused on \textit{"cyclic orderings"} of the \textit{half edges} around each vertex, and gluings of those half edges to obtain a connected graph $\Gamma$. There is another way of constructing cellular embeddings which comes from polygon presentations of surfaces. This is likely more familiar to the reader, and therefore more intuitive. The question might be asked, "why not just use this more typical example." One answer would be, the former construction is actually quite standard in the literature on combinatorial maps. A better answer however is, the combinatorics and the notation involved in the previous (clockwise) \textit{"cyclic vertex ordering"} construction is much more convenient for later constructions involving \textit{medial quivers}, \textit{surface algebras}, and the representation theory that follows. It will be useful, and sometimes more intuitive to have this second construction though. Let us begin with the following recipe:

\begin{enumerate}
\item Write $\phi$ as a product of disjoint cycles $\phi_1 \phi_2 \cdots \phi_p$
\item To each cycle $\phi_i$ of length $m_i$ we associate a $m_i$-gon, oriented \textit{counterclockwise}. 
\item Then we glue the sides of each polygon according to $\alpha$ so that the sides which are glued have opposite orientation. 
\item From this gluing we obtain a cyclic order of edges $\sigma = \phi^{-1}\alpha$ around each vertex. Note: $\alpha = \alpha^{-1}$ since it is required to be an involution. Also, the vertices with cyclic orderings are the corners of the polygons after gluing.
\end{enumerate}

\begin{ex}
Let us once again illustrate by example. Take $C = [\sigma, \alpha, \phi]$ from the previous construction where 
\[ \sigma = \sigma_x \sigma_y = (1,2,3)(4,5,6), \quad \alpha = \alpha^1\alpha^2\alpha^3 = (1,4)(2,5)(3,6).\]
This implies $\phi = (162435)$. This is represented by a counterclockwise oriented hexagon:\\
\
\\
\[ \begin{tikzpicture}
  [thick
  ,decoration =
    {markings
    ,mark=at position 0.5 with {\arrow{stealth'}}
    }
  ]
  \newdimen\R
  \R=2.7cm
  \node {$\phi$};
  \draw[-stealth'] (-150:{0.7*\R}) arc (-150:150:{0.7*\R});
  \foreach \x/\l in
    { 60/f,
     120/a,
     180/b,
     240/c,
     300/d,
     360/e
    }
    \draw[postaction={decorate}] ({\x-60}:\R) -- node[auto,swap]{\l} (\x:\R);
  \foreach \x/\l/\p in
    { 60/{x}/above,
     120/{y}/above,
     180/{x}/left,
     240/{y}/below,
     300/{x}/below,
     360/{y}/right
    }
    \node[inner sep=2pt,circle,draw,fill,label={\p:\l}] at (\x:\R) {};
\end{tikzpicture}\]
\\
\
\\

The \textit{"word"} associated to the polygon given by $\phi$ in most standard texts containing material on polygon presentations of surfaces is 
\[ abcdef \leftrightarrow (162435). \]
The gluing $\alpha = (1,4)(2,5)(3,6)$, then says we must glue the faces:
\[ a \leftrightarrow d, \quad b \leftrightarrow e, \quad c \leftrightarrow f.\]
Care must be taken to glue sides so that their orientations "appose" one another so that the surface obtained is \textit{oriented} according to the \textit{counterclockwise} oriented face. The cellularly embedded graph that we obtain lives on a torus $\TT^2$. We can determine this purely via the combinatorics by computing
\[ \chi(\Sigma) = |\phi| - |\alpha| + |\sigma| = |\phi|-|\Gamma_1|+|\Gamma_0|= 1-3+2 = 0,\]
and since $\chi(\Sigma) = 2g(\Sigma)-2$ we have that the genus of $\Sigma$ is $g = 1$. \\
\
\\

\begin{center}
\fbox{\includegraphics[
    page=1,
    width=280pt,
    height=280pt,
    keepaspectratio
]{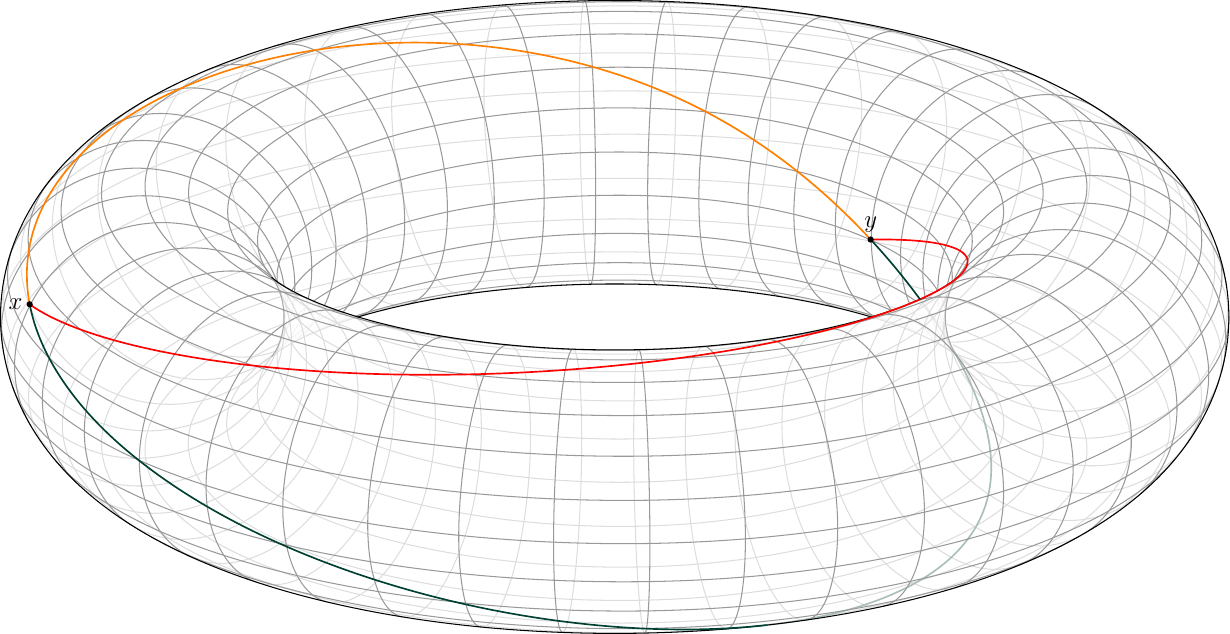}}
\end{center}
\end{ex}

\medskip

\subsection{Hypermaps and Bipartite Maps}

It is important to not here that if one wishes to construct bipartite maps so as the restrict to the case of dessins d'enfants, one may proceed in two ways. We may modify $C = [\sigma, \alpha, \phi]$ by adding $2$-cycles to the permutation $\sigma$ which will corespond to subdividing the edges given by $\sigma$, and then redefine $\alpha$ accordingly. The second way is to allow arbitrary $3$-constellations and not just those with $\alpha$ a fixed-point free involution. This can be interpreted as $\sigma$ being the rotation of edges around black vertices, $\alpha$ being the rotation of edges around white vertices, and $(\sigma \alpha) = \phi$ the rotation around the faces. One may then relate this to the theory of triangle groups and Fuchsian groups. This is not used here, but will become important in later work and should be noted. For these constructions the reader is referred to \cite{JW} \S 2.1.2 and \S 3, and to \cite{LZ} \S 1.5. One might also look at \cite{S4} \S 5.

\section{Medial Quivers of Combinatorial Maps and Constellations}\label{Medial Quivers}
There is a very natural way of associating a cellularly embedded graph to a quiver, and a quiver to a cellularly embedded graph. In particular, we can define a bijection of such objects. 

\begin{defn} The way we do this is by choosing the quiver to be the directed medial graph of the cellularly embedded graph. In particular, for each face $\phi_j$ of $C=[\sigma, \alpha, \phi]$, we place a vertex on the interior of each edge of the boundary of $\phi_j$. We then connect the vertices counter-clockwise with arrows. This forms the \textbf{medial quiver} of the constellation, or equivalently of the cellularly embedded graph.
\end{defn}

\begin{ex}
\noindent
\includegraphics[
    page=1,
    width=200pt,
    height=200pt,
    keepaspectratio
]{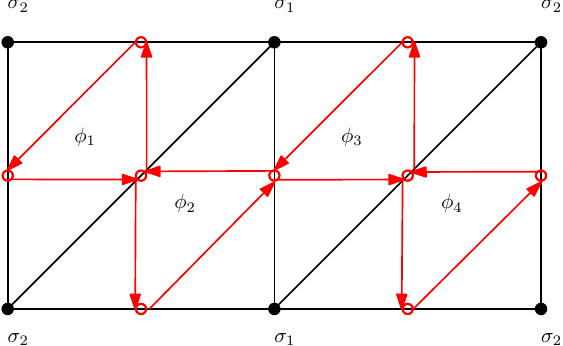}
\\
\
\\
\
\\
As an example, we have the medial quiver for a triangulation of a torus given by the constellation
\[ \phi = (1,2,3)(4,5,6)(7,8,9)(10,11,12), \quad \alpha = (1,5)(2,12)(3,4)(6,7)(8,10)(9,11 ). \]
It has face cycles given by $\phi = \phi_1 \phi_2 \phi_3 \phi_4$, and the gluing $\alpha$ identifies the top edges and bottom edges, as well as the left and right side edges, in the typical way. 
\end{ex}

\section{Surface Algebras}\label{Surface Algebras}

We now introduce the \textit{surface algebras}, which along with their $\fm$-adic completions will be the main objects of study in what follows. 

Let $Q = (Q_0, Q_1, h, t)$ be a quiver, with the set of vertices $Q_0$, and the set of arrows $Q_1$. There are two maps, 
\[ t, h: Q_1 \to Q_0 \]
taking an arrow $a \in Q_1$ to its \textbf{head} $ha$, and \textbf{tail} $ta$. This is a refinement of the incidence map for an undirected graph, and we define 
\[ \partial a = \{\partial_{\bullet}a, \partial^{\bullet}a\}:= \{ta, ha\}.\] 
In this case the order is not arbitrary as it would be for undirected graphs. The \textbf{path algebra} of a quiver $Q$, denoted $kQ$, over a field $k$, is the $k$-vector space spanned by all oriented paths in $Q$. It is an associative algebra, and is finite dimensional as a $k$-vector space if and only if $Q$ has no oriented cycles. There are trivial paths $i \in Q_0$, given by the vertices, and multiplication in the path algebra is defined by concatenation of paths, when such a concatenation exists. Otherwise the multiplication is defined to be zero. More precisely, if $p$ and $q$ are directed paths in $Q$, and $hp=tq$, then $qp$ is defined as the concatenation of $p$ and $q$. Note, we will read paths from \emph{right to left}. Let $A=kQ$. The \textbf{vertex span} $A_0=k^{Q_0}$, and the \textbf{arrow span} $A_1=k^{Q_1}$ are finite dimensional subspaces. $A_0$ is a finite dimensional commutative $k$-algebra, and $A_1$ is an $A_0$-bimodule. The path algebra then has a grading by path length, 
\[ A = A_0 \la A_1 \ra = \bigoplus_{d=0}^{\infty} A^{\otimes d}.\]
The path algebra $A$ has primitive orthogonal idempotents $\{e_i\}_{i \in Q_0}$. Let $A_{i,j} = e_jAe_i$ be the $k$-linear span of paths in $Q$, from vertex $i$ to $j$. Let $\fm = \prod_{d=1}^{\infty} A^{\otimes d}$ denote the \emph{arrow ideal} of $Q$, generated by the arrows $Q_1$. We will define the \textbf{complete path algebra} to be
\[ \cA = A_0 \la \la A_1 \ra \ra = \prod_{d=0}^{\infty} A^{\otimes d} .\]
We put the $\fm$-adic topology on $\cA$, with neighborhoods of $0$ generated by $\fm^n$. The elements of $\cA$ are all formal linear combinations of paths, including infinite linear combinations. If $\phi: \cA \to \cA$ is an automorphism fixing $\cA_0$ then $\phi$ is continuous in the $\fm$-adic topology, and $\fm$ is invariant under such algebra automorphisms.

\begin{defn}
An \textbf{ideal} in the path algebra $A$ will be a two sided ideal generated by linear combinations of paths which share a common starting vertex and terminal vertex in the quiver. The \textbf{quotient path algebra} of a quiver with relations will be the quotient by this ideal. 
\end{defn}

Next let us turn to the specific quivers with relations of interest for our current purposes. 

\begin{defn}
We will define a \textbf{free surface algebra} to be the path algebra of the medial quiver of any combinatorial map $C=[\sigma, \alpha, \phi]$.
\end{defn}

\begin{defn}\label{gentle surface algebra}
	Let $Q = (Q_0, Q_1)$ be a finite connected quiver. Then we say the bound path algebra $\Lambda = kQ/I$ is a \textbf{surface algebra} if the following properties hold:
	
	\begin{enumerate}
		\item For every vertex $x \in Q_0$ there are exactly two arrows $a, a' \in Q_1$ with $ha = x = ha'$, and exactly two arrows $b, b \in Q_1$ such that $tb = x = tb'$.
		\item For any arrow $a \in Q_1$ there is exactly one arrow $b \in Q_1$ such that $ba \in I$, and there is exactly one arrow $c \in Q_1$ such that $ac \in I$.
		\item For any arrow $a \in Q_1$ there is exactly one arrow $b' \in Q_1$ such that $b'a \notin I$, and there is exactly one arrow $c' \in Q_1$ such that $ac' \notin I$.
		\item The ideal $I$ is generated by paths of length $2$. 
	\end{enumerate}
\end{defn}

These will be called \textbf{gentle relations}. Such quivers with relations are a very popular class of path algebras studied in the representation theory of associative algebras. It is important to note that given a quiver such that every vertex has in-degree and out-degree exactly $2$, we may choose several ideals $I$ such that $kQ/I$ is a (gentle) surface algebra. Such algebras are always infinite dimensional, but they retain many of the nice combinatorial and representation theoretic properties of finite dimensional gentle algebras. See for example the preprint \cite{CB4} and references therein for more background on the representation theory of infinite dimensional string algebras (a class of algebras which includes gentle surface algebras). The next Theorem indicates how this project started. While attempting to count certain indecomposable modules over a class of algebras, it was noticed that associating a graph to the algebras in an essentially unique way, one could apply Polya theory with some minor success. Later, it was discovered that the graph constructed for these purposes was simply a dessin d'enfant. In particular, 

\begin{theorem}
There is a unique dessin $C=[\sigma, \alpha, \phi]$ associated to each surface algebra $A$ such that the following properties hold:
\begin{enumerate}
\item The quiver $Q$ of the surface algebra is the directed medial graph of $C=[\sigma, \alpha, \phi]$.
\item The cycles of the permutation $\phi$ are in one-to-one correspondence with cycles of (gentle) zero relations $I$, as described by Definition \ref{gentle surface algebra}.
\item One can define an action of $\phi = \phi_1 \phi_2 \cdots \phi_s$, on the cycles of relations $I=\la I(\phi_1), I(\phi_2), ..., I(\phi_s) \ra$ as a cyclic permutation on the arrows of each cycle of relations $I(\phi_j)$. 
\item One can partition the arrows of the quiver $Q$ with respect to $\sigma$ such that the nonzero simple cycles of the quiver $Q$ are in one-to-one correspondence with cycles $\sigma_i$ of $\sigma$. 
\item One may define an action of $\sigma$ on the partition $Q_1 = \{ c(\sigma_1), c(\sigma_2), ..., c(\sigma_r) \ra \}$ and thus on the arrow ideal $\fm$ of $A$, and on the $k$-vector space $k^{Q_0}$ giving the arrow span of $A$. The action is again by cyclically permuting the arrows in each non-zero cycle $c_i = c(\sigma_i)$. 
\item The permutation $\alpha$ determines how one may glue nonzero cycles, or equivalently cycles of relations in order to obtain a (gentle) surface algebra. 
\item Let $Q_1(\sigma) = \{c(\sigma_1), c(\sigma_2), ..., c(\sigma_r)\}$ be the partition of $Q_1$ with respect to $\sigma$, and let $Q_1(\phi) = \{I(\phi_1), I(\phi_2), ..., I(\phi_s)\}$ be the partition with respect to $\phi$. Let
\[ k^{Q_1(\sigma)} = V_1 \oplus V_2 \oplus \cdots \oplus V_r \]
where $V_i = k^{c(\sigma_i)}$, and let 
\[ k^{Q_1(\phi)} = W_1 \oplus W_2 \oplus \cdots \oplus W_s \]
where $W_j = k^{I(\phi_j)}$. Then and action of the absolute Galois group $\cG(\overline{Q}/Q)$ on the dessin $C=[\sigma, \alpha, \phi]$ induces an action on the surface algebra and its quiver. In particular, it induces an automorphism of the vector spaces $k^{Q_1(\sigma)}$ and $k^{Q_1(\phi)}$, such that $\dim_kV_i = \dim_k g \cdot V_i$, and $\dim_kW_j = \dim_k g \cdot W_j$ for any $g \in \cG(\overline{Q}/Q)$. 
\end{enumerate}
\end{theorem}

\begin{proof}
The proof is simple once one has seen the definition of the medial quiver and the definition of a surface algebra is then very natural. It is reminiscent of the construction of the Brauer tree for blocks of group algebras in modular representation theory. We simply note that the ideal of gentle relations can be partitioned into disjoint cycles in the quiver. They may overlap themselves. Likewise the nonzero cycles can be partitioned. Once this is done, if one simply defines $\sigma$ and $\phi$ to be the permutations of $Q_1$ giving these cycles, one gets a dessin. The permutation $\alpha$ can be computed from $\sigma \alpha \phi = \id$, and it simply tells us how to glue the cycles (possibly to themselves in some places). The statements concerning the actions of $\sigma$ and $\phi$ and the induced automorphisms will be refined in the following sections and proven there. They will correspond to information given by the center and noncommutative normalization of the complete surface algebra, and the pullback diagram defining a dessin order. 
\end{proof}

Let us look briefly at a few more examples. 

\subsection{The Dihedral Ringel Algebra $\tilde{A}(1)$}
\begin{ex}
Let $\sigma = (1,2), \alpha = (1,2), \phi = (1)(2)$. Then the closed surface algebra $\Lambda(\mathfrak{c})$ given by the constellation $\mathfrak{c}_1 = [\sigma, \alpha, \varphi]_0$ is given by the graph with one vertex and one loop embedded in the sphere. In particular $\Lambda(\mathfrak{c}) = \tilde{A}(1)$ is given by the quiver
\[ \xymatrix{\bullet \ar@(ul,dl)_x \ar@{.>}@(ur,dr)^y} \]
and is isomorphic to $k\la x,y \ra / \la x^2, y^2 \ra$. This is a classic example from the representation theory of associative algebras. See for example Ringel's work \cite{R1}. The genus in this case is then zero. 
\end{ex}

\subsection{$\tilde{A}(2)$}
\begin{ex}
Let $\mathfrak{c}_2=[\sigma, \alpha, \varphi]_2$ be given by $\sigma = (1,4)(2,3) \alpha = (1,3)(2,4), \varphi = (1,2)(3,4)$, then $\chi(\mathfrak{c}_2) = 2$ and $g(\mathfrak{c}_2)=0$. The embedded graph can be represented by the equator of the sphere with two vertices on it. The quiver which comes from this graph is\\
\[
\begin{tikzcd}
1 \arrow[rr, "a_1", bend left=49] \arrow[rr, "b_1"', dashed, bend right=49] &  & 2 \arrow[ll, "a_2" description, bend right] \arrow[ll, "b_2" description, dashed, bend left]
\end{tikzcd}
\]

\end{ex}

\subsection{$\tilde{A}(3)$}
\begin{ex}
Let $\mathfrak{c}_3=[\sigma, \alpha, \varphi]_3$ be defined by $\sigma = (1,6,2,4,3,5), \alpha = (1,4)(2,5)(3,6), \varphi = (1,2,3)(4,5,6)$. Then we have $\chi(\mathfrak{c}_3) = 1-3+2=0$ so $g(\mathfrak{c}_3)=1$. The graph embedded on the torus can be obtained by a gluing of the square to obtain the torus,

the quiver is then, 
\[
\xymatrix{
	& 1 \ar@/^1pc/[dr] \ar@{.>}[dr] & \\
	3 \ar@/^1pc/[ur] \ar@{.>}[ur] & & 2 \ar@/^1pc/[ll] \ar@{.>}[ll]
}
\]

\end{ex}

\subsection{$\tilde{A}(4)$}
\begin{ex}
Let $\mathfrak{c}_4 = [\sigma, \alpha, \varphi]_4$ be given by
$\sigma = (1,8,3,6)(2,5,4,7), \alpha = (1,5)(2,6)(3,7)(4,8), \varphi = (1,2,3,4)(5,6,7,8)$
$\chi(\mathfrak{c}_4)=2-4+2$ so $g(\mathfrak{c}_4) = 1$.

\[
\xymatrix{
	1 \ar@/^1pc/[r] \ar@{.>}[r] &   2 \ar@/^1pc/[d] \ar@{.>}[d] \\
	4 \ar@/^1pc/[u] \ar@{.>}[u] &   3 \ar@/^1pc/[l] \ar@{.>}[l]
}
\]
\end{ex}



\section{Dessin Orders and Surface Orders}\label{Dessin Orders}
In this section we define the \emph{Dessin Orders} and \emph{Surface Orders} which are naturally associated to a dessin, or equivalently to its constellation or surface algebra, and more generally to any graph cellularly embedded in a Riemann surface. These are particularly useful in defining invariants for the action of the absolute Galois group on Dessins. The methods are adapted from the classical theory of $R$-orders, a standard tool used in modular representation theory of finite groups (see for example \cite{B1, B2} for an introduction and \cite{R1}, \cite{RR}, \cite{Ro1, Ro2} for examples of applications outside of modular representation theory).

\subsection{Definitions and Properties}
\begin{defn}
	Let $R$ be a commutative ring, which will generally be a Noetherian domain or a Dedekind domain \footnote{A Dedekind domain $R$ is an integral domain such that every localization $R_{\frp}$ at a prime ideal $\frp$ is a discrete valuations ring. Equivalently, every finitely generated torsion free module is projective, or $R$ is integrally closed and of dimension $\leq 1$. See for example \cite{S1} Proposition $4$, pg. 10} with field of fractions $\KK$, and for a maximal ideal $\fm$ it has residue field $k = R/\fm$. We will generally only work with rings of algebraic integers $\mathscr{O}_K$, coming from field extension of $K/\QQ$, the polynomial ring $\FF[x]$ over a field $\FF$, or localizations and completions of these. 
\end{defn}
	
\begin{defn}
An $R$-\textbf{lattice} is a finitely generated projective module over $R$. In particular, if $R$ is a Dedekind domain, every $R$-lattice is finitely generated and torsion free. 
\end{defn}

\begin{ex}
For example, if $R = \ZZ$, then $\ZZ^2$ is a $\ZZ$-lattice via addition of ordered pairs. As another example, let $R = \CC[x, y]/(xy)$. Then $\CC[x.y]/(xy)$ is an $R$-lattice over itself via the action given by multiplication by $\overline{x}$ and $\overline{y}$, the residues of $x$ and $y$ in $R$. One can visualize this via the maps of the bigraded shifts
\[
\xymatrix{
& R \ar[dl]_{x} \ar[dr]^y & \\
R(-1,0) \ar[d]_x & & R(0,-1) \ar[d]^y \\
R(-2,0) \ar[d]_x & & R(0,-2) \ar[d]^y \\
R(-3,0) \ar[d]_x & & R(0,-3) \ar[d]^y \\
\vdots & & \vdots
}
\]

\end{ex}
	
\begin{defn}
An $R$-\textbf{Order} $\Lambda$ in a $k$-algebra $A$ is a unital subring of $A$ such that
	\begin{enumerate}
		\item $\KK \Lambda = A$, and
		\item $\Lambda$ is finitely generated as an $R$-module. 
	\end{enumerate}
\end{defn}

\begin{defn}
	Let $C = [\sigma, \alpha, \phi]$ be a constellation, and let $\Gamma \hookrightarrow X$ be the associated graph cellularly embedded in the compact Riemann surface $X$. Further, let $n(i)=n_i = |\sigma_i|$ denote the length of the cycle $\sigma_i$ in the permutation $\sigma = \sigma_1 \sigma_2 \cdots \sigma_p$. Remember, for a constellation $C$, and the associated graph $\Gamma$, the length of the (nonzero) cycle in the medial quiver $Q(C)$ with gentle relations, which is associated to $\sigma_i$ (and its corresponding vertex of $\Gamma$) is just the order of the cycle $\sigma_i$.
	\begin{enumerate}
		\item For each cycle $\sigma_i$, corresponding to the vertex $i \in \Gamma_0$, we associate a Dedekind domain $R_i$, with a maximal ideal $\fm_i$, and a \textbf{vertex} ($R_i$-)\textbf{order}, $\Lambda(\sigma_i) = \Lambda_i$. 
		\item The \emph{vertex order} associated to the cycle $\sigma_i$ is then given by a matrix $R$-subalgebra of $\Mat_{n_i \times n_i}(R_i)$. 
		\[
		\Lambda_i = \begin{pmatrix}
			R_i & \fm_i & \fm_i & \cdots & \fm_i & \fm_i \\
			R_i & R_i   & \fm_i & \cdots & \fm_i & \fm_i \\
			R_i & R_i   & R_i   & \cdots  & \fm_i & \fm_i \\
	   \vdots & \vdots & \vdots & \ddots & \vdots & \vdots \\
	   R_i & R_i & R_i & \cdots & R_i & \fm_i \\
	   R_i & R_i & R_i & \cdots & R_i & R_i 
		\end{pmatrix}_{n(i)}
		\]
		In general, $\Lambda_i$ is a \emph{hereditary order} if $R_i$ is local. \footnote{Recall, an algebra, or an order, is \textbf{hereditary} if no module has a minimal projective resolution greater than length one. This means the \textit{projective dimension} of any module is no greater than one, and therefore the global dimension of the algebra is at most one.}.		
		
		\item Let $\Lambda_i^{(k,k)}$ denote the $(k,k)$ entry of $\Lambda_i$ (in $R_i$).  
		\item For each $1 \leq k \leq n_i$ let
		\[ 
		P_{i,1}:= \begin{pmatrix}
		R_i \\ \vdots \\ R_i \\ R_i \\ R_i \\ \vdots \\ R_i \\ R_i
		\end{pmatrix},		 
		\ P(\sigma_i) = \sigma_i \cdot P_{i,1}:= \begin{pmatrix}
		\fm_i \\ R_i \\ \vdots \\ R_i \\ R_i \\ \vdots \\ R_i \\ R_i
		\end{pmatrix}, \cdots,
		\ P(\sigma_i^{k-1}) = P_{i,k}= \begin{pmatrix}
		\fm_i \\ \vdots \\ \fm_i \\ R_i \\ R_i \\ \vdots \\ R_i \\ R_i
		\end{pmatrix}, \cdots ,
		P_{i,n_i}:= \begin{pmatrix}
		\fm_i \\ \vdots \\ \fm_i \\ \fm_i \\ \fm_i \\ \vdots \\ \fm_i \\ R_i
		\end{pmatrix} 
		\]
		where the $k^{th}$ entry is the first entry equal to $R_i$ for $P_{i,k} = \sigma_i^{k-1} \cdot P_{i,1} = P(\sigma_i^{k-1})$. 
	\end{enumerate}
\end{defn}

The modules $\{P_{i,k}: \ 1 \leq k \leq n_i\}$ give a complete set of non-isomorphic indecomposable projective (left) $\Lambda_i$-modules, with the natural inclusions
\[ P_{i,1} \hookleftarrow P_{i,2} \hookleftarrow \cdots \hookleftarrow P_{i,n_i-1} \hookleftarrow P_{i,n_i} 
\hookleftarrow P_{i,1}. \]
where the final map is given by left-multiplication by $\fm_i$. If we identify $P_{i,k}$ with the edge 
$e_k^i = e_k(\sigma_i)$, where $\sigma_i = (e_1^i, e_2^i, ..., e_{n_i}^i)$ is a cyclic permutation, then the chain of 
inclusions can be interpreted in terms of the cycle $\sigma_i$. From the embedding $\Gamma
\hookrightarrow X$ given by the constellation $C = [\sigma, \alpha, \phi]$, this can be interpreted as 
walking clockwise around the vertex of $\sigma_i$. We will take $P_{i,k} = P_{i,k+n_i}$, and each $e_k^i$ is 
multiplied by the automorphism $\sigma_i$, i.e. there is some multiplication by a 
power of $\fm_i$ involved. 
\indent In particular, conjugation\footnote{In any number field, Dedekind domain, or complete discrete valuation ring primes are invertible.} by 
\[ \sigma_i= \begin{pmatrix}
	0 	  & 0 & 0 & \cdots & 0 & \fm_i \\
	1 	  & 0 & 0 & \cdots & 0 & 0 \\
	0 	  & 1 & 0 & \cdots & 0 & 0 \\
\vdots & \vdots & \vdots & \ddots & \vdots & \vdots \\
	0 	  & 0 & 0 &\cdots & 0 & 0 \\
	0 	  & 0 & 0 &\cdots & 1 & 0 
\end{pmatrix}_{n_i} \]
cyclically permutes the indecomposable projective $\Lambda_i$-modules $P_{i,k}$, and it induces an automorphism of the matrix algebra $\Lambda_i$ which we also call $\sigma_i$. Now, for each pair of cycles $\sigma_i, \sigma_j \in S_{[2m]}$ of $\sigma$, we fix an 
isomorphism
\[ R_i/\fm_i \cong R_j/\fm_j. \] 
Identifying all such rings, let $k = R_i/\fm_i$ for all $\sigma_i \in \Gamma_0$. Let $
\pi_i: R_i \to k$ be a fixed epimorphism with kernel $\fm_i$ a maximal ideal of $R_i$. Now, we have a pull-back 
diagram
\[ 
\xymatrix{
R_{i,j} \ar[r]^{\tilde{\pi}_i}  \ar[d]_{\tilde{\pi}_j} & R_i \ar[d]^{\pi_i} \\
R_j \ar[r]_{\pi_j} & k 
}
\]
which is in general different and non-isomorphic for different choices of $\pi_i$ and $\pi_j$.

\begin{defn}
	Let $\cN(\Lambda) = \prod_{\sigma_i \in \Gamma_0} \Lambda_i$. Let $e_k^i$ be an edge around $\sigma_i$, and let $
	\alpha^{i,j}_{k,l} = (e_k^i, e_l^j)$ be a $2$-cycle of the fixed-point free involution $\alpha$ of $C=[\sigma, \alpha, \phi]$ giving the end vertices $
	\sigma_i$ and $\sigma_j$ of the edge $e_k^i \equiv e_l^j$ under the gluing identifying the half-edges $e_k^i$ and 
	$e_l^j$. It is possible that $\sigma_i=\sigma_j$ if $\alpha^{i,j}_{k,l}$ defines a loop at the vertex $\sigma_i$ in $
	\Gamma$. We replace the product $\Lambda_i^{(k,k)} \times \Lambda_j^{(l,l)}$ in $\Lambda_i \times \Lambda_j$ with $R_{i,j}$. 
	This identifies the $(k,k)$ entry of $\Lambda_i$ with the $(l,l)$ entry of $\Lambda_j$, modulo $\fm_{i,j}$, the ideal of $R_{i,j}$ given via the pullback of $\fm_i$ and $\fm_j$. Doing this for all 
	edges of $\Gamma$, we get the \textbf{Surface Order} $\Lambda := 
	\Lambda(C) = \Lambda(\Gamma)$ associated to the constellation $C$, or equivalently to the embedded 
	graph $\Gamma \hookrightarrow X$. We will call $\Lambda$ a \textbf{Dessin Order} if the constellation $C$ gives a dessin. We will call the hereditary order $\prod_{\sigma_i}\Lambda_i$ the \textbf{normalization} of the  order $\Lambda$. 
\end{defn}

\begin{prop}
%
		The indecomposable projective $\Lambda$-modules are in bijection with the $2$-cycles $(e^i_k, e^j_l)=\alpha^{i,j}_{k,l}$, of $\alpha \in S_{2m}$ for the constellation $C=[\sigma, \alpha, \phi]$. Equivalently, the indecomposable projectives are in bijection with the edges $\Gamma_1$. We label them as $P_{e}$ for $\alpha^{i,j}_{k,l} = e = (e^i_k, e^j_l) \in \Gamma_1$ attached to the vertices $\sigma_i$ and $\sigma_j$. 
\end{prop}

\section{The Gel'fand-Ponomarev Algebra}\label{Gel'fand-Ponomarev}

\begin{ex}
In the highly influential paper \cite{GP}, Gel'fand and Ponomarev investigated the indecomposable representations of the Lorentz group, which is equivalent to classifying the "\emph{Harish-Chandra modules}" of the Lie algebra $\sln_2$.  They used methods now standard in the representation theory of so-called \emph{string algebras}. Since $\sl_2(\CC)$ is the Lie algebra of both the Lorentz group and $\SL_2(\CC)$, they study representations of $\sln_2(\CC)$, as well as the Lie algebra $\su_2$ of the maximal compact subgroup $\SU_2 \subseteq \SL_2(\CC)$. Note, $\sln_2(\CC) = \su_2 \otimes \CC$. In \cite{FH} one sees the following diagram for $\sln_2(\CC)$
\[ \xymatrix{
\cdots \ar[r]_{X} & V_{\alpha-4} \ar@(ul,ur)^{H} \ar[r]_{X} \ar@/_/[l]_{Y} & V_{\alpha-2} \ar@(ul,ur)^{H} \ar[r]_{X} \ar@/_/[l]_{Y} & V_{\alpha} \ar[r]_{X} \ar@/_/[l]_{Y} \ar@(ul,ur)^{H}  & V_{\alpha+2} \ar@(ul,ur)^{H} \ar[r]_{X} \ar@/_/[l]_{Y} & \ar@/_/[l]_{Y} \cdots
}\]
with the typical basis, 
\[ H_+ = \begin{pmatrix}
 0 & 1 \\
 0 & 0
 \end{pmatrix}, \quad H_-=\begin{pmatrix}
 0 & 0 \\
 1 & 0
 \end{pmatrix}, \quad H_3 = \begin{pmatrix}
 1 & 0 \\
 0 & -1
 \end{pmatrix}\]
In their study of Harish-Chandra modules\footnote{It seems very likely under the construction given here that the Bruhat-Tits Building for $\SL_2$, described by Serre in \cite{S4}, in II \S 1.1, is closely related to the universal cover of the surface algebras being a directed four regular tree, and the Gel'fand-Ponomarev algebras being related to an amalgam of two free (noncommutative) associative algebras on two generators, and being closely related to the Lie algebra $\sln_2(\CC)$.}, 
they asked the following question:
\emph{Let $k$, be a field and let $V_1$ and $V_2$ be finite dimensional $k$-vector spaces. Suppose we have nilpotent operators 
\[ H_+ : V_1 \to V_2, \quad H_{-}: V_2 \to V_1, \quad H_3: V_2 \to V_2 \]
such that $YZ = 0 = ZX$. Fix a basis of $V_1$ and $V_2$ and classify all canonical forms of $H_{\pm}$ and $H_3$. 
}
 If we take instead
 \[ V = V_1 \oplus V_2, \quad X = \begin{pmatrix}
 H_{-}  & 0 \\
 0 & H_{+} 
 \end{pmatrix}, \quad Y = \begin{pmatrix}
 0 & H_3 \\
 0 & 0
 \end{pmatrix} \]
then this is equivalent to classifying all representations of the following path algebra:

$I=\la xy, yx \ra$, $\Lambda = kQ/I$, 

\[ \xymatrix{
	\bullet \ar@(ul,dl)_{x} \ar@(dr,ur)_{y}
} \]
\\
The surface algebra is then $k\la  x, y \ra  /\la xy, yx \ra\cong k[x,y]/(xy)$. We may then take the dessin order to be $\Lambda = R_{1,2} = k[x,y]/(xy)$ which has normalization $R_1 \times R_2 = k[x] \times k[y]$, with maximal ideals $\fm_1 = (x)$ and $\fm_2 = (y)$ respectively. Fixing an isomorphism $R_1/\fm_1 = k \to k = R_2/\fm_2$, we get a gluing
\[ \xymatrix{
R_{1,2} = k[x,y]/(xy) \ar[r] \ar[d] & k[y] \ar[d] \\
k[x]  \ar[r] & k
}\]
The completion of the path algebra $\Lambda = kQ/I$ is then $\widehat{\Lambda} = k[[x,y]]/(xy)$.
\end{ex}

\section{Some Examples}\label{Examples}

\begin{ex}
	Let $R_i = k[[x_i]]$ have maximal ideal $\mathfrak{m}_i = (x_i)$, with residue field $k=R_i/\mathfrak{m}_i$. 
	Let $\Gamma$ be the genus zero graph,  
	\[
	\xymatrix{
	\sigma_1 \ar@{-}[r]^{\alpha^{1}} & \sigma_2 \ar@{-}[r]^{\alpha^{2}} & \sigma_3 \ar@{-}[r]^{\alpha^{3}} & \sigma_4
	}
	 \]
	 given by the constellation $C=[\sigma, \alpha, \phi]$ such that 
	 $\sigma = (1)(2,3)(4,5)(6) = \sigma_1 \sigma_2 \sigma_3 \sigma_4$, and $\alpha = (1,2)(3,4)(5,6) = (e_1)(e_2)(e_3)$. We have then that 
	 \[ \cN(\Lambda) = \left\lbrace \begin{pmatrix}
		 R_1
	 \end{pmatrix}, \begin{pmatrix}
		 R_2 & \fm_2 \\
		 R_2 & R_2
	 \end{pmatrix}, \begin{pmatrix}
	 R_3 & \fm_3 \\
	 R_3 & R_3
	 \end{pmatrix}, \begin{pmatrix}
	 	R_4 
	 \end{pmatrix}\right\rbrace = \prod_{i=1}^4 \Lambda_i \]
This is contained in the matrix algebra
\[ \Mat_{1\times 1}(k[[x_1]]) \times \Mat_{2 \times 2}(k[[x_2]]) \times  \Mat_{2\times 2}(k[[x_3]]) \times  \Mat_{1\times 1}(k[[x_4]]). \]
We then have the congruences of diagonal entries modulo 
	 \[ \mathfrak{m} = (x_ix_j)_{i \neq j} = (x_1x_2, x_1x_3, x_1x_4, x_2x_3, x_2x_4, x_3x_4)\]
We then have 
\[ \Lambda = 
\begin{pmatrix}
R_{1,2} & \fm_2 & 0 \\
R_2 & R_{2,3} & \fm_3 \\
0 & R_3 & R_{3,4} \\
\end{pmatrix} 
\]
We have the following diagrams
\[ \xymatrix{
&& k[[x_1,x_2]]/(x_1x_2) \ar[r] \ar[d] & k[[x_1]] \ar[d] \\
&k[[x_2,x_3]]/(x_2x_3) \ar[r] \ar[d] & k[[x_2]] \ar[d] \ar[r] & k \\
k[[x_3,x_4]]/(x_3x_4) \ar[r] \ar[d] & k[[x_3]] \ar[d] \ar[r] & k & \\
k[[x_4]] \ar[r] & k & & 
}\]
which gives us 
\[ R_{1,2} = k[[x_1,x_2]]/(x_1x_2), \ R_{2,3} = k[[x_2,x_3]]/(x_2x_3), \ R_{3,4} = k[[x_3,x_4]]/(x_3x_4).\]
Notice, the equalities in $k=R_i/\mathfrak{m}_i$ (i.e. equalities of residues modulo $\mathfrak{m}_i$) are given by $\alpha$ in the constellation $C = [\sigma, \alpha, \phi]$. We then get a pullback diagram
\[
\xymatrix{
\Lambda \ar[rr]^{\pi_1} \ar[d]_{\iota_1} & & \Lambda/\rad(\Lambda)= k^3  \ar[d] \\
\cN(\Lambda) \ar[rr]_{\pi_2} & & \cN(\Lambda)/\rad(\cN(\Lambda) = k^6
}
\]
where $\cN(\Lambda) = \prod_{i=1}^4 \Lambda_i$, $\Lambda_i = k[[x_i]]$ for $i=1,...,4$ and $\fm_i = (x_i) = \rad(k[[x_i]])$, so that $\rad(\Lambda_1) = k = \rad(\Lambda_4)$ and $\rad(\Lambda_2) = k \times k = \rad_k(\Lambda_3)$; and $\rad(\Lambda)) = k \times k \times k$. The automorphisms given by $\sigma_2$ and $\sigma_3$ on $\Lambda_2$ and $\Lambda_3$ are
\[ \begin{pmatrix}
	0 & (x_2) \\
	1 & 0
\end{pmatrix}, \quad \text{and} \quad \begin{pmatrix}
	0 & (x_3) \\
	1 & 0
\end{pmatrix} \]

We can think of this as a gluing of the matrix algebra $\cN(\Lambda)$ as in the following diagram over $\cN(\Lambda)$

\[ \cN(\Lambda) = \left\lbrace \begin{pmatrix}
		 R_1
	 \end{pmatrix}, \begin{pmatrix}
		 R_2 & \fm_2 \\
		 R_2 & R_2
	 \end{pmatrix}, \begin{pmatrix}
	 R_3 & \fm_3 \\
	 R_3 & R_3
	 \end{pmatrix}, \begin{pmatrix}
	 	R_4 
	 \end{pmatrix}\right\rbrace = \prod_{i=1}^4 \Lambda_i \]

\begin{tikzcd}
\fbox{$R_1$} \arrow[rr, "\alpha_1" description, no head] &  & \fbox{$R_2$} \arrow[rd, "\sigma_2" description, no head] & (x_2) &  & \fbox{$R_3$} \arrow[rd, "\sigma_3" description, no head] & (x_3) &  &  \\
 &  & R_2 & \fbox{$R_2$} \arrow[rru, "\alpha_2" description, no head] &  & R_3 & \fbox{$R_3$} \arrow[rr, "\alpha_3" description, no head] &  & \fbox{$R_4$}
\end{tikzcd}

Further, this is exactly the completion of the surface algebra associated to the graph
\[
\xymatrix{
	\sigma_1 \ar@{-}[r]^{\alpha^{1}} & \sigma_2 \ar@{-}[r]^{\alpha^{2}} & \sigma_3 \ar@{-}[r]^{\alpha^{3}} & \sigma_4
}
\]
In particular, we have the following quiver with relations 
\[I = \la ba, cb, dc, ed, fe, af \ra ,\]
\[ \xymatrix{
	\bullet_x \ar@(ul,dl)_{a} \ar@/_/[r]_b  &   \bullet_y \ar@/_/[r]_d \ar@/_/[l]_g  &   \bullet_z \ar@/_/[l]_f  \ar@(dr,ur)_{e}
}\]
\end{ex}

\begin{ex}\label{torus order diagram}
Let us look at the torus dessin given by $C = [\sigma, \alpha, \phi]$ where
\[ \sigma = (1,2,3)(4,5,6), \quad \alpha= (1,4)(2,5)(3,6) \]

\begin{center}
\fbox{\includegraphics[
    page=1,
    width=280pt,
    height=280pt,
    keepaspectratio
]{torus.pdf}}
\end{center}

We have the automorphisms of $\Lambda_i$ and $\Lambda_2$:
\[ 
\begin{pmatrix} 
0 & 0 & \sigma_1 \\ 
1 & 0 & 0 \\ 
0 & 1 & 0 
\end{pmatrix} \quad \quad \quad 
\begin{pmatrix} 
0 & 0 & \sigma_2 \\ 
1 & 0 & 0 \\ 
0 & 1 & 0 
\end{pmatrix} \]

we may represent the action of $\sigma \alpha = \phi^{-1}$ as the product
\[
\begin{pmatrix} 
0 & 0 & \sigma_1 & 0 & 0 & 0 \\
1 & 0 & 0 & 0 & 0 & 0 \\ 
0 & 1 & 0 & 0 & 0 & 0 \\
0 & 0 & 0 & 0 & 0 & \sigma_2 \\ 
0 & 0 & 0 & 1 & 0 & 0 \\ 
0 & 0 & 0 & 0 & 1 & 0 
\end{pmatrix} 
\begin{pmatrix}
0 & 0 & 0 & 1 & 0 & 0 \\
0 & 0 & 0 & 0 & 1 & 0 \\
0 & 0 & 0 & 0 & 0 & 1 \\
1 & 0 & 0 & 0 & 0 & 0 \\
0 & 1 & 0 & 0 & 0 & 0 \\
0 & 0 & 1 & 0 & 0 & 0 \\
\end{pmatrix} \ \ \bullet \ \
\begin{pmatrix}
k[[x_1]]  & (x_1) & (x_1) & 0 & 0 & 0 \\
k[[x_1]] & k[[x_1]] & (x_1) & 0 & 0 & 0 \\
k[[x_1]] & k[[x_1]] & k[[x_1]] & 0 & 0 & 0 \\
0 & 0 & 0 & k[[x_2]] & (x_2) & (x_2) \\
0 & 0 & 0 & k[[x_2]] & k[[x_2]] & (x_2) \\ 
0 & 0 & 0 & k[[x_2]] & k[[x_2]] & k[[x_2]]
\end{pmatrix}
\]
This gives the following diagram, 
\[\xymatrix{
k[[x_1]] \ar@/^/@{-}[rrrr] & (x_1) & (x_1) & & k[[x_2]] & (x_2) & (x_2) \\
k[[x_1]] & k[[x_1]] \ar@/^/@{-}[rrrr] & (x_1) & & k[[x_2]] & k[[x_2]] & (x_2) \\ 
k[[x_1]] & k[[x_1]] & k[[x_1]] \ar@/^/@{-}[rrrr] & & k[[x_2]] & k[[x_2]] & k[[x_2]]
}\]

If we allow $\sigma_2$ to act on $\Lambda_2$ via left multiplication, 
\[ 
\begin{pmatrix} 
0 & 0 & \sigma_2 \\ 
1 & 0 & 0 \\ 
0 & 1 & 0 
\end{pmatrix} \ \bullet \ 
\begin{pmatrix}
k[[x_2]] & (x_2) & (x_2) \\
k[[x_2]] & k[[x_2]] & (x_2) \\ 
k[[x_2]] & k[[x_2]] & k[[x_2]]
\end{pmatrix}
\]

this permutes the projective modules over $\Lambda_2$
\[ P(\sigma_2) \to P(\sigma_2^2) \to P(\sigma_2^3) \to P(\sigma_2) \]
This gives a new gluing via $\alpha = (1,5)(2,6)(3,4)$ corresponding to the permutation matrix
\[
\begin{pmatrix}
0 & 0 & 0 & 0 & 1 & 0 \\
0 & 0 & 0 & 0 & 0 & 1 \\
0 & 0 & 0 & 1 & 0 & 0 \\
0 & 0 & 1 & 0 & 0 & 0 \\
1 & 0 & 0 & 0 & 0 & 0 \\
0 & 1 & 0 & 0 & 0 & 0 \\
\end{pmatrix}
\]

we may represent the action of $\sigma \alpha = \phi^{-1}$ as the product
\[
\begin{pmatrix} 
0 & 1 & 0 & 0 & 0 & 0 \\
0 & 0 & 0 & 0 & 0 & 0 \\ 
\sigma_1 & 0 & 0 & 0 & 0 & 0 \\
0 & 0 & 0 & 0 & 1 & 0 \\ 
0 & 0 & 0 & 0 & 0 & 1 \\ 
0 & 0 & 0 & \sigma_2 & 0 & 0 
\end{pmatrix} 
\begin{pmatrix}
0 & 0 & 0 & 0 & 1 & 0 \\
0 & 0 & 0 & 0 & 0 & 1 \\
0 & 0 & 0 & 1 & 0 & 0 \\
0 & 0 & 1 & 0 & 0 & 0 \\
1 & 0 & 0 & 0 & 0 & 0 \\
0 & 1 & 0 & 0 & 0 & 0 \\
\end{pmatrix} \ \ \bullet \ \
\begin{pmatrix}
k[[x_1]]  & (x_1) & (x_1) & 0 & 0 & 0 \\
k[[x_1]] & k[[x_1]] & (x_1) & 0 & 0 & 0 \\
k[[x_1]] & k[[x_1]] & k[[x_1]] & 0 & 0 & 0 \\
0 & 0 & 0 & k[[x_2]] & (x_2) & (x_2) \\
0 & 0 & 0 & k[[x_2]] & k[[x_2]] & (x_2) \\ 
0 & 0 & 0 & k[[x_2]] & k[[x_2]] & k[[x_2]]
\end{pmatrix}
\]
This can be pictured via the new diagram, 
\[\xymatrix{
k[[x_1]] \ar@/^/@{-}[rrrrrd] & (x_1) & (x_1) & & k[[x_2]] & (x_2) & (x_2) \\
k[[x_1]] & k[[x_1]] \ar@/^/@{-}[rrrrrd] & (x_1) & & k[[x_2]] & k[[x_2]] & (x_2) \\ 
k[[x_1]] & k[[x_1]] & k[[x_1]] \ar@/^/@{-}[rruu] & & k[[x_2]] & k[[x_2]] & k[[x_2]]
}\]

\end{ex}

Suppose we have a dessin $[\sigma, \alpha, \phi]$, with $\sigma = \sigma_1 \cdots \sigma_p$. The projective modules may be represented by their \textbf{Loewy diagrams}
\[
\xymatrix{
                                            & x \ar[dl]_{\sigma_r} \ar[dr]^{\sigma_r} & \\
\sigma_r x \ar[d]_{\sigma_r} &                                                            & \sigma_s x \ar[d]^{\sigma_s} \\
\sigma_r^2 x \ar[d]_{\sigma_r} &                                                        & \sigma_s^2 x \ar[d]^{\sigma_s} \\
\sigma_r^3 x \ar[d]_{\sigma_r} &                                                         & \sigma_s^3 x \ar[d]^{\sigma_s} \\
\vdots \ar[d]_{\sigma_r} &                                                                   & \vdots \ar[d]^{\sigma_s} \\
\fbox{$\sigma_r^{\lcm(|\sigma_r|, |\sigma_s|)} x = x \ar[d]^{\sigma_s} $}&      & \fbox{$\sigma_s^{\lcm(|\sigma_r|, |\sigma_s|)} x = x \ar[d]^{\sigma_s}$} \\
                                              \vdots &                                                & \vdots 
}\]
where $x \in Q_0$ is some vertex in the quiver of the dessin, or equivalently some edge in the graph $\Gamma$. The two "arms" of the projective correspond to two cycle of edges around two vertices of $\Gamma$ starting at the edge $x \in \Gamma_1$. Equivalently, they correspond to two nonzero cycles in the quiver starting at the vertex $x \in Q_0$. 

Notice, the two arms meet back up at $"x"$ after $\lcm(|\sigma_r|, |\sigma_s|)$ many applications of $\sigma_r$ and $\sigma_s$. This is not to imply that they are identified with one another. In fact, there is a geometric interpretation of the module category (as well as the derived category) for surface algebras that represents these modules as paths on the corresponding Riemann surface (see for example \cite{BS}). Let us illustrate this by example. Take the algebra and dessin order given by the torus dessin $C = [\sigma, \alpha, \phi]$ where
\[ \sigma = (1,2,3)(4,5,6), \quad \alpha= (1,4)(2,5)(3,6) \]
Then the dessin order has normalization
\[ \cN(\Lambda) = \Lambda_1 \times \Lambda_2 = 
\left\lbrace\begin{pmatrix}
 R_1 & \fm_1 & \fm_1 \\
 R_1 & R_1 & \fm_1 \\
 R_1 & R_1 & R_1 
\end{pmatrix}
\ \ \times \ \
\begin{pmatrix}
 R_2 & \fm_2 & \fm_2 \\
 R_2 & R_2 & \fm_2 \\
 R_2 & R_2 & R_2 
\end{pmatrix} \right\rbrace \subset \Mat_{3 \times 3}(R_1) \times \Mat_{3 \times 3}(R_2)
\]
and the dessin order is given by gluing the diagonals as in Example \ref{torus order diagram}. The indecomposable projective modules of $\Lambda_1$ and $\Lambda_2$ are
\[
\left\lbrace \begin{pmatrix}
R_1 \\ R_1 \\ R_1 
\end{pmatrix}, \quad 
\begin{pmatrix}
\fm_1 \\ R_1 \\ R_1 
\end{pmatrix}, \quad
\begin{pmatrix}
\fm_1 \\ \fm_1 \\ R_1 
\end{pmatrix} \right\rbrace \quad \text{and} \quad 
\left\lbrace \begin{pmatrix}
R_2 \\ R_2 \\ R_2 
\end{pmatrix}, \quad 
\begin{pmatrix}
\fm_2 \\ R_2 \\ R_2 
\end{pmatrix}, \quad
\begin{pmatrix}
\fm_2 \\ \fm_2 \\ R_2 
\end{pmatrix} \right\rbrace
\]

In terms of the surface algebra, we have the projectives
\[ 
P(1,4) = \begin{pmatrix}
& (1,4) & \\ 
2 & & 5 \\ 
3 & & 6 \\
1 & & 4 \\ 
2 & & 5 \\ 
3 & & 6 \\
\vdots & & \vdots
\end{pmatrix}, \quad 
P(2,5) = \begin{pmatrix} 
& (2,5) & \\ 
3 & & 6 \\
1 & & 4 \\ 
2 & & 5 \\ 
3 & & 6 \\
1 & & 4 \\
\vdots & & \vdots
\end{pmatrix}, \quad 
P(3,6) = \begin{pmatrix}  
& (3,6) & \\
1 & & 4 \\ 
2 & & 5 \\ 
3 & & 6 \\
1 & & 4 \\
2 & & 5 \\
\vdots & & \vdots
\end{pmatrix} \]
with uniserial radicals $U(1), ..., U(6)$ given by the left and right column of the Loewy series for the projectives in each matrix.\footnote{This description is useful for generalizing and applying the results of \cite{RZ} and \cite{Zi} in order to understand the Artin braid group action on dessins.}

For concreteness of examples, we will take the dessin order $\Lambda$ associated to a constellation $C=[\sigma, \alpha, \phi]$ to be the \emph{completed path algebra} of the surface algebra associated to $C$.

\section{Basic Invariants of $\cG(\overline{Q}/Q)$}\label{Basic Invariants}
In this section we compute the center of dessin orders and surface algebras, as well as the (noncommutative) normalizations. We then prove that these are invariant under the action of the absolute Galois group on dessins. Let $C = [\sigma, \alpha, \phi]$ and let $\Lambda$ be the associated dessin order. Denote by $\cZ(\Lambda)$ the \textit{center}, and $\cN(\Lambda)$ the \textit{normalization}.

\begin{theorem}\label{basic invariants main theorem}
Suppose two constellations $C = [\sigma, \alpha, \phi]$ and $C' = [\sigma', \alpha', \phi']$ lie in the same orbit under the action of $\cG(\overline{Q}/Q)$. Further, let $\Lambda$ and $\Lambda'$ be their associated (completed) surface algebras. Then the following isomorphisms hold.
\begin{enumerate}
\item $\cZ(\Lambda) \cong \cZ(\Lambda')$
\item $\cN(\Lambda) \cong \cN(\Lambda')$
\end{enumerate}
\end{theorem}

\begin{lemma}
Suppose \[A = kQ/I = A_0 \la A_1 \ra = \bigoplus_{d=0}^{\infty} A^{\otimes d} \] 
is the path algebra of the constellation $C=[\sigma, \alpha, \phi]$, where $\sigma = \sigma_1 \sigma_2 \cdots \sigma_r$. It has primitive orthogonal idempotents $\{e_i\}_{i \in Q_0}$. Let $A_{i,j} = e_jAe_i$ be the $k$-linear span of paths in $Q$, from vertex $i$ to $j$. Let $\fm = \prod_{d=1}^{\infty} A^{\otimes d}$ denote the \emph{arrow ideal} of $Q$, generated by the arrows $Q_1$. We have then that the \textbf{complete path algebra} is
\[ \Lambda = A_0 \la \la A_1 \ra \ra = \prod_{d=0}^{\infty} A^{\otimes d} \]
Then 
\begin{enumerate}
\item \[ \cZ(\Lambda) \cong A_0 \cdot k[[z_1, z_2, ..., z_r]]/(z_iz_j)_{i \neq j} \]
and
\item \[  \cN(\Lambda) \cong \prod_{\substack{\sigma_i \\ i=1,...,r}} \overline{k\tilde{\mathbb{A}}_{|\sigma_i|}^{eq}} \]
where $\overline{k\tilde{\mathbb{A}}_{|\sigma_i|}^{eq}}$ is the completion of the hereditary algebra given by the quiver path algebra, 
\[k\tilde{\mathbb{A}}_{|\sigma_i|}^{eq}:~
\vcenter{\hbox{  
		\begin{tikzpicture}[point/.style={shape=circle, fill=black, scale=.3pt,outer sep=3pt},>=latex]
		\node[point,label={above:$x_1$}] (0) at (0,0) {};
		\node[point,label={right:$x_2$}] (1) at (1.5,-.5) {};
		\node[point,label={left:$x_{|\sigma_i|}$}] (n) at (-1.5,-.5) {};
		\node[point,label={below:$x_{k}$}] (2) at (0,-4) {};
		\node[point,label={right:$x_{k-1}$}] (3) at (1.5,-3.5) {};
		\node[point,label={left:$x_{k+1}$}] (4) at (-1.5,-3.5) {};
		
		
		\path[->]
		(4) edge [dashed,bend left=45]  (n)
		(1) edge [dashed,bend left=45]  (3)
		(0) edge [bend left=15] node[midway, above] {$a_1$} (1)
		(3) edge [bend left=15] node[midway, below] {$a_{k}$} (2)
		(2) edge [bend left=15] node[midway, below] {$a_{k+1}$} (4)
		(n) edge [bend left=15] node[midway, above] {$a_{|\sigma_i|}$} (0);
		\end{tikzpicture} 
}}
\]
\end{enumerate}
\end{lemma}

\begin{proof}
\begin{enumerate}
\item Let $\sigma_i$ be the cycle of $\sigma$ associated to the $i^{th}$ nonzero cycle in the quiver $Q/I$. Say $\sigma_i$ 
corresponds to the arrows $\{a_1, a_2, ..., a_{n(i)}\}$, where $n_i = |\sigma_i|$. Choosing a distinguished arrow, 
say $a_1$, let $\sigma_i^k$ be identified with $\mathfrak{c}_k= a_k a_{k-1} \cdots a_1a_na_{n-1} \cdots a_{k+1}
$, the cyclic permutation of the arrows in the cycle of $\sigma_i$. Let $z_i = \sum_{k=1}^{n_i} \mathfrak{c}_k$. 
Then $z_i$ commutes with any arrow $b \in Q_1$. Indeed, 
\begin{align*}
bz_i &= b(\mathfrak{c}_1 + \mathfrak{c}_2 + \cdots + \mathfrak{c}_{n_i}) \\
	   &= b\mathfrak{c}_1 + b\mathfrak{c}_2 + \cdots + b\mathfrak{c}_{n_i} 
\end{align*}
and 
\begin{align*}
b\mathfrak{c}_k &= b a_k a_{k-1} \cdots a_1a_na_{n-1} \cdots a_{k+1} \\
						 &\neq 0 \iff ha_k=tb, b \in \sigma_i \\
						 & \quad \iff b=a_{k+1}.
\end{align*}
From this we gather $b \mathfrak{c}_k =  \mathfrak{c}_k b$ and therefore $bz_i = z_ib$. Thus, the subalgebra
\[ k \la \la z_1, z_2, ..., z_r \ra \ra  \subset \Lambda, \]
is commutative with all paths in $Q/I$, $z_iz_j = 0$ if and only if $i \neq j$, and so $\cZ(\Lambda) \cong A_0 \cdot k[[z_1, z_2, ..., z_r]]/(z_iz_j)_{i \neq j}$. 
\item This follows from the definitions. In particular, for $k\tilde{\mathbb{A}}_{|\sigma_i|}^{eq}$, the completion is given by the hereditary order $\Lambda_i$. 
\end{enumerate}
\end{proof}

\begin{proof}
(\textbf{Proof of Theorem} \ref{basic invariants main theorem}): First, the cycle types of the constellation $C=[\sigma, \alpha, \phi]$ are known invariants of the action of the absolute Galois group. From the structure of $\cZ(\Lambda)$ and $\cN(\Lambda)$ given in the definition of dessin orders, normalizations, and from the previous Lemma, we can now see that $\cZ(\Lambda)$\footnote{This will be important when studying gluings of central characters and their restrictions given by representations in the BGG category $\mathcal{O}$ of the semi-simple Lie algebras given by local surface orders.} and $\cN(\Lambda)$ are invariants of this action on $\Lambda$ as well. 
\end{proof}

Alone, this does not seem to provide any significant new results for dessins aside from a reinterpretation of the cycle types of $[\sigma, \alpha, \phi]$ into representation theoretic language. We will study how these invariants along with projective resolutions of simple modules recover the dessin entirely and how each encodes the information of $[\sigma, \alpha, \phi]$. These invariants give some interesting implications in the representation theory of the surface algebras and dessin orders as well. 

\begin{ex}
Take $\Lambda$ to be the completion of the surface algebra from the following quiver with relations 
\[I = \la ba, cb, dc, ed, fe, af \ra ,\]
\[ \xymatrix{
	\bullet_x \ar@(ul,dl)_{a} \ar@/_/[r]_b  &   \bullet_y \ar@/_/[r]_d \ar@/_/[l]_g  &   \bullet_z \ar@/_/[l]_f  \ar@(dr,ur)_{e}
}\]
The we have $\cZ(\Lambda) \cong A_0 \cdot k[[z_1, z_2, z_3, z_4]]/(z_iz_j)_{i \neq j}$, and 
\[ \cN(\Lambda) \cong \overline{k\tilde{\mathbb{A}}_{1}^{eq}} \times \overline{k\tilde{\mathbb{A}}_{2}^{eq}} \times \overline{k\tilde{\mathbb{A}}_{2}^{eq}} \times \overline{k\tilde{\mathbb{A}}_{1}^{eq}}.
\]

\end{ex}

\section{Projective Resolutions of Simple Modules}\label{Resolutions}
We now turn to some more complicated results. Here we will prove that projective resolutions of simple modules over the dessin order $\Lambda$ completely recover the dessin, without any other information required. We give an explicit description of such resolutions and we show classifying dessin orders with given normalization is equivalent to classifying dessins with given monodromy group.\footnote{The results in this section may be used to obtain results on Beilinson's conjectures using the Serre-Swann Theorem and a mixture of algebraic and topological K-theory. This is studied in ongoing work.}

In this section we denote by $P^{\bullet}$ a complex of projective modules
\[ \cdots \to P^{-1} \to P^0 \to P^{1} \to \cdots \]
over some algebra $\Lambda$. Let $\Lambda = \Lambda(C)$ be an order given by the constellation $C=[\sigma, \alpha, \phi]$. 

\begin{theorem}
	\begin{enumerate}
		\item The indecomposable projective modules $P_{e} = P_{\alpha^{i,j}_{k,l}}$ have radical
		\[ \rad(P_e) = U(\sigma_i^{k+1}) \oplus U(\sigma_j^{l+1}) \]
		where $U(\sigma_p^q) \cong P_{p,q} \in \Mod(\cN(\Lambda_p))$ is an indecomposable uniserial $\Lambda$-module and an indecomposable projective $\cN(\Lambda)$-module.
		\item The minimal projective resolution of the simple module $S(\alpha^{i,j}_{k,l}) = S(e^i_k, e^j_l)$ of $\Lambda$, corresponding to the vertex $\alpha^{i,j}_{k,l}=(e^i_k, e^j_l)$ of $Q$ (or equivalently the edge of the same labeling in $\Gamma(C)$ connecting vertex $\sigma_i$ and $\sigma_j$ in $\Gamma$), is infinite periodic. In particular the period $p$ of the minimal resolution $P^{\bullet}(\alpha^{i,j}_{k,l}) = P^{\bullet} \to S(\alpha^{i,j}_{k,l})$ is exactly the least common multiple,
		\[ p(P^{\bullet}(i,j))  = \lcm\{|\cO_{\phi}(e^i_k)|, |\cO_{\phi}(e^j_l)|\}.\]
		where $\cO_{\phi}(e^i_k)$ and $\cO_{\phi}(e^j_l)$ are the orbits under the action of $\phi$ of $e^i_k$ and $e^j_l$ on the two anti-cycles (or relations in $I$) passing through the vertex $\alpha^{i,j}_{k,l}$. 
		\item The differentials in the minimal projective resolution of the simple module $S(\alpha^{i,j}_{k,l})$, $d^m: P^{m} \to P^{m+1}$,
		\[ P^{\bullet}(\alpha_{k,l}^{i,j}): \cdots \to P^{m} \to P^{m+1} \to \cdots \]
		are given by multiplication by the matrix 
		\[ d^m:= \begin{pmatrix}
			a(\phi^{m} \cdot e^i_k)  & 0 \\
			0 & a(\phi^{m} \cdot e^j_l)
		\end{pmatrix}. \]
		where $a(\phi^{m}e^i_k) \in Q_1$ is the arrow with $ta = \phi^{m}e^i_k$ and $ta(\phi^{m} \cdot e^j_l) = \phi^{m} \cdot e^j_l$. 
		\item The syzygies $\Omega^m(\alpha_{k,l}^{i,j})= \ker(d^m)$ are of the form
		\[ \Omega^m(\alpha_{k,l}^{i,j}) = U(\phi^{m} e^i_k)) \oplus U(\phi^{m} e^j_l)) \]
		The uniserial modules at the vertex $\phi^{m} e^i_k$ and $\phi^{m} e^j_l$ which are annihilated by left multiplication by the arrows associated to $P(\phi_i^{m-1}\cdot e^i_k) \to P(\phi_i^{m}\cdot e^i_k)$ and $P(\phi_i^{m-1}\cdot e^j_l) \to P(\phi_i^{m}\cdot e^j_l)$ by definition of the relations $I$. 
\end{enumerate}
\end{theorem}

This will be useful later when explaining how to recover a graph embedded in a Riemann surface entirely in terms of the projective resolutions of the simple modules. 




\begin{proof}
	\begin{enumerate}
		\item First, $\alpha^{i,j}_{k,l} = (e^i_k, e^j_l)=e$, and with fixed labeling of the edges of $\Gamma(C)$, we have $e^i_k = \sigma_i^{k-1}\cdot e^i_1$, and $e^j_l = \sigma_j^{l-1} \cdot e_1^j$, given by the automorphism 
		\[ \sigma_i^k:=
		\begin{pmatrix}
		0 & 0 & \cdots & 0 & \sigma_i \\
		1 & 0 & \cdots & 0 & 0 \\
		\vdots & \vdots & \ddots & \vdots & \vdots \\
		0 & 0 & \cdots & 0 & 0 \\
		0 & 0 & \cdots & 1 & 0 \\
		\end{pmatrix}^k
		\]
		So, $\sigma_i$ acts on the algebra $\Lambda$ by left multiplication of $e^i_{j-1}$ (and therefore $a^i_{j-1}$) by 
		the arrow $\sigma_i a^i_{j-1}=a^i_j$ in the quiver $Q(\Lambda)$. Notice, this multiplication is always 
		nonzero since $\sigma_i = (e_1^i, e_2^i, ..., e_{n_i}^i)$ is a cyclic permutation around the vertex it corresponds to 
		in $\Gamma(C)$, and there is by definition a unique arrow by which $\sigma$ acts on a given idempotent $e^i_{j-1}
		$ 
		(and on $a^i_{j-1}$) lying on this cycle corresponding to the hereditary order $\Lambda_i$ in the pullback diagram 
		defining $\Lambda(C)$. 
		\item Let $P(\alpha_{k,l}^{i,j})$ be the projective cover of $S(\alpha_{k,l}^{i,j})$. From the desription of the radical of 
		$P(\alpha_{k,l}^{i,j})$ as the two uniserial modules in $\cN(\Lambda)$ corresponding to the idempotents $\sigma \cdot e^i_k$ 
		and $\sigma \cdot e^j_l$ in $\Lambda_i$ and $\Lambda_j$ respectively, the next term in the resolution is the direct sum of 
		the two indecomposable projective covers $P(\phi \cdot e^i_k)$ and $P(\phi \cdot e^j_l)$ in $\Mod(\Lambda)
		$. Clearly the kernel of the covering $P(\phi \cdot e^j_l) \to U(\phi \cdot e^j_l)$ is exactly the uniserial 
		$U(\phi^2 \cdot e^j_l)$, and its projective cover is 
		$P(\phi^{2} \cdot e^j_l)$. The kernel of this covering is $U(\phi^{3} \cdot 
		\cdot e^j_l)$. This pattern continues also for $\phi e^i_k$, and the terms 
		$P^m$ in the resolution are
		\[ P(\phi^{m}  e^i_k) \oplus P(\phi^{m} e^j_l). \]
		So the terms have indecomposable direct summands which cycle through the orbit of $e^i_k$ and $e^j_l$ under the 
		action of $\phi$. The orbits are anti-cycles in $I$, the ideal of relations of the surface algebra, and the place at which the two cycle meet up at $\alpha^{i,j}_{k,l} 
		= (e^i_k, e^j_l)$ is exactly $p = \lcm\{|\cO_{\phi}(e^i_k)|, |\cO_{\phi}(e^j_l)|$.
		\item Since the kernel of the cover of a uniserial $P(\alpha^{i,j}_{k,l}) \to U(e^i_k)$ is exactly $U(\sigma e^j_l)$ and 
		it is embedded in $P(\alpha^{i,j}_{k,l})$ as a submodule via multiplication by the arrow $a: ha=\sigma e^i_k$, we get that the differential is indeed, 
		\[ d^m:= \begin{pmatrix}
		a(\phi^{m}\cdot e^i_k)  & 0 \\
		0 & a(\phi^{m} \cdot e^j_l)
		\end{pmatrix}. \]
		\item This now follows from $(1)-(3)$. 
	\end{enumerate}
\end{proof}

\begin{corollary}\label{fixed 3-constellation to fixed normalization}
	Classifying all $3$-constellations with a fixed cycle types $[\lambda_1, \lambda_2, \lambda_3]$ is equivalent to classifying all \emph{dessin orders} with the same normalization, or equivalently all \emph{surface algebras} with the same cycle decomposition with respect to $\sigma$ or $\phi$. In particular, the resolutions of simple modules over $\Lambda$ completely encode the information given by $[\sigma, \alpha, \phi]$. The normalization completely encodes $\sigma$, and the pullback diagram completely encodes the information given by $\alpha$. 
\end{corollary}

\begin{ex}
	Let $C=[\sigma, \alpha, \phi]$ be given by 
	\[ \sigma = (1,4)(2,3), \quad \alpha = (1,3)(2,4),\quad  \phi = (1,2)(3,4),\] then $\chi(C) = 2$ and $g(C)=0$. The embedded graph can be represented by the equator of the sphere with two vertices on it. Or, if we embed it in the plane:
	\[
	\begin{tikzcd}
	&  & \phi_2=(3,4) \\
	\sigma_2=(2,3) \arrow[rr, "{\alpha_1=(1,3)}", no head, bend left=49] & \phi_1=(1,2) & \sigma_1 = (1,4) \arrow[ll, "{\alpha_2=(2,4)}", no head, bend left=49]
	\end{tikzcd}
	\]
	The quiver which comes from this graph is\\
	\[
	\begin{tikzcd}
	\alpha_1=(1,3) \arrow["a_1", dd, bend left] \arrow["b_1", dd, dashed, bend right=71] \\
	\\
	\alpha_2=(2,4) \arrow["a_2",uu, bend left] \arrow["b_2", uu, dashed, bend right=71]
	\end{tikzcd}
	\]
	The associated matrix data is
	\[ \Lambda = \begin{pmatrix}
	k[[x]] & (x) & 0 & \\
	k[[x]] & k[[x,y]]/(xy) & (y) \\
	0 		& k[[y]] & k[[y]] 
	\end{pmatrix} \]
	With normalization
	\[ 
	\cN(\Lambda) =  \left\lbrace \begin{pmatrix}
	\lambda_{11} & x\cdot \lambda_{12} \\
	\lambda_{21} & \lambda_{22} 
	\end{pmatrix} \times \begin{pmatrix}
	\mu_{11} & x\cdot \mu_{12} \\
	\mu_{21} & \mu_{22} 
	\end{pmatrix}\Bigg| \ \lambda_{ij} \in k[[x]],\ \ \mu_{kl} \in k[[y]] \right\rbrace
	\]
	The projective resolution of the simple $S(\alpha_1)$ has the following form
	\[ \xymatrix{
	\cdots \ar[r] &  \ar[r]^{\begin{pmatrix}
		a_1 & 0 \\
		0 & b_1
		\end{pmatrix}} P(\alpha_2) \oplus P(\alpha_2) & \ar[r]^{\begin{pmatrix}
		a_2 & 0 \\
		0 & b_2
		\end{pmatrix}} P(\alpha_1) \oplus P(\alpha_1)& \ar[r]^{\begin{pmatrix}
		a_1 & 0 \\
		0 & b_1
		\end{pmatrix}} P(\alpha_2) \oplus P(\alpha_2) & P(\alpha_1) \ar[r] & S(\alpha_1)
	}\]
\end{ex}

\section{More Examples of Resolutions of Simple Modules}\label{More Examples}
Let us illustrate once more by a few examples. 

\begin{ex}
We will compute the center and normalization, and describe some of the projective resolutions of simple modules of several dessin orders corresponding to two dessins in the the same orbit of $\cG(\overline{\QQ}/\QQ)$, and two others in a different orbit. 

We first compute the projective resolutions of the simple modules, which give us the information for $\phi$ as well as $\alpha$. Let us use the following shorthand for the indecomposable projective module $P(e)$ corresponding to an edge of $\Gamma$, we identify $e = (e^i_k, e^j_l) = \alpha^{i,j}_{k,l}$. Let $C_1 = [\sigma, \alpha, \phi]$, where

\begin{itemize}
\item $\sigma = (1)(2,3)(4,5)(6,7,8,9)(10)(11)(12,13)(14)$
\item $\alpha = (1,2)(3,4)(5,6)(7,10)(8,11)(9,12)(13,14)$
\end{itemize}

\medskip
\medskip

\begin{tikzcd}
\sigma_8 & \sigma_7 \arrow[l, no head] &  &  &  &  \\
 & \sigma_5 & \sigma_4 \arrow[r, no head] \arrow[lu, no head] \arrow[l, no head] \arrow[ld, no head] & \sigma_3 \arrow[r, no head] & \sigma_2 \arrow[r, no head] & \sigma_1 \\
 & \sigma_6 &  &  &  & 
\end{tikzcd}

\medskip
\medskip

and let $S(e) = S(\alpha^1) = S(1,2)$ be the simple module at the edge $e=(1,2)$. Then the resolution has the following form
\medskip

\begin{center}
\begin{tikzcd}[row sep=18pt, column sep=15pt]
{(1,2)} &  & {(13,14) \quad (13,14)} \arrow[lldddd] &  & {(7,10) \quad (7,10)} \arrow[lldddd] &  & {\fbox{$(1,2) \quad (3,4)$}} \arrow[lldddd] \\
{\fbox{$(1,2) \quad (3,4)$}} \arrow[u] &  & {(13,14) \quad (9,12)} \arrow[u] &  & {(7,10) \quad (5,6)} \arrow[u] &  &  \\
{(3,4) \quad (5,6)} \arrow[u] &  & {(9,12) \quad (8,11)} \arrow[u] &  & {(5,6) \quad (3,4)} \arrow[u] &  &  \arrow[uu, dotted] \\
{(5,6) \quad (9,12)} \arrow[u] &  & {(8,11) \quad (8,11)} \arrow[u] &  & {(3,4) \quad (1,2) } \arrow[u] &  &  \\
{(9,12) \quad (13,14)} \arrow[u] &  & {(8,11) \quad (7,10)} \arrow[u] &  & {(1,2) \quad (1,2)} \arrow[u] &  & 
\end{tikzcd}
\end{center}

\medskip
We also have the following gluing diagram of matrix algebras.
\medskip

\begin{center}
\begin{tikzcd}[row sep=18pt, column sep=15pt]
 &  &  & R_8 &  &  &  &  &  &  \\
 &  & R_7 \arrow[rd, "\sigma_7" description] & \mathfrak{m}_7 &  &  &  &  &  &  \\
 &  & R_7 & R_7 \arrow[uu, no head, bend right] &  &  &  &  &  &  \\
R_5 & R_4 \arrow[rd, "\sigma_4" description] \arrow[l, no head] & \mathfrak{m}_4 & \mathfrak{m}_4 & \mathfrak{m}_4 &  &  &  &  &  \\
 & R_4 & R_4 \arrow[rd, "\sigma_4^2" description] \arrow[uuu, no head, bend left] & \mathfrak{m}_4 & \mathfrak{m}_4 & R_3 \arrow[rd, "\sigma_3" description] & \mathfrak{m}_3 & R_2 \arrow[rd, "\sigma_2" description] & \mathfrak{m}_2 &  \\
 & R_4 & R_4 & R_4 \arrow[rd, "\sigma_4^3" description] \arrow[rru, no head] & \mathfrak{m}_4 & R_3 & R_3 \arrow[ru, no head] & R_2 & R_2 \arrow[r, no head] & R_1 \\
 & R_4 & R_4 & R_4 & R_4 &  &  &  &  &  \\
 &  &  &  & R_6 \arrow[u, no head] &  &  &  &  & 
\end{tikzcd}
\end{center}
\medskip
\medskip
\medskip
\medskip
\medskip

For $C_2 = [\sigma, \alpha, \phi]$, where
\begin{itemize}
\item $\sigma = (1)(2,3)(4,5)(6,7,8,9)(10)(11)(12,13)(14)$
\item $\alpha = (1,2)(3,4)(5,6)(7,12)(8,11)(13,14)(9,10)$
\end{itemize}
we have 
\medskip

\begin{tikzcd}
 & \sigma_6 &  &  &  &  \\
 & \sigma_5 & \sigma_4 \arrow[r, no head] \arrow[lu, no head] \arrow[l, no head] \arrow[ld, no head] & \sigma_3 \arrow[r, no head] & \sigma_2 \arrow[r, no head] & \sigma_1 \\
\sigma_8 & \sigma_7 \arrow[l, no head] &  &  &  & 
\end{tikzcd} 
\medskip

Now, the resolution of the simple module $S(1,2)$ has the following form, 

\medskip

\begin{center}
\begin{tikzcd}[row sep=18pt, column sep=15pt]
{(1,2)} &  & {(7,10) \quad (8,11)} \arrow[lldddd] &  & {(13,14) \quad (9,12)} \arrow[lldddd] &  & {\fbox{$(1,2) \quad (3,4)$}} \arrow[lldddd] \\
{\fbox{$(1,2) \quad (3,4)$}} \arrow[u] &  & {(8,11) \quad (8,11)} \arrow[u] &  & {(9,12) \quad (5,6)} \arrow[u] &  &  \\
{(3,4) \quad (5,6)} \arrow[u] &  & {(8,11) \quad (9,12)} \arrow[u] &  & {(5,6) \quad (3,4)} \arrow[u] &  &  \arrow[uu, dotted] \\
{(5,6) \quad (7,10)} \arrow[u] &  & {(9,12) \quad (13,14)} \arrow[u] &  & {(3,4) \quad (1,2) } \arrow[u] &  &  \\
{(7,10) \quad (7,10)} \arrow[u] &  & {(13,14) \quad (13,14)} \arrow[u] &  & {(1,2) \quad (1,2)} \arrow[u] &  & 
\end{tikzcd}
\end{center}

\medskip
The corresponding gluing diagram of matrix algebras is:
\medskip

\begin{center}
\begin{tikzcd}[row sep=18pt, column sep=15pt]
 & R_6 &  &  &  &  &  &  &  &  \\
 & R_4 \arrow[rd, "\sigma_4" description] \arrow[u, no head] & \mathfrak{m}_4 & \mathfrak{m}_4 & \mathfrak{m}_4 &  &  &  &  &  \\
R_5 & R_4 & R_4 \arrow[rd, "\sigma_4" description] \arrow[ll, no head, bend left] & \mathfrak{m}_4 & \mathfrak{m}_4 & R_3 \arrow[rd, "\sigma_2" description] & \mathfrak{m}_3 & R_2 \arrow[rd, "\sigma_2" description] & \mathfrak{m}_2 &  \\
 & R_4 & R_4 & R_4 \arrow[rd, "\sigma_4" description] \arrow[rru, no head] & \mathfrak{m}_4 & R_3 & R_3 \arrow[ru, no head] & R_2 & R_2 \arrow[r, no head] & R_1 \\
 & R_4 & R_4 & R_4 & R_4 &  &  &  &  &  \\
 &  &  &  & R_7 \arrow[u, no head] \arrow[rd, "\sigma_7" description] & \mathfrak{m}_7 &  &  &  &  \\
 &  &  & R_8 & R_7 & R_7 \arrow[ll, no head, bend left] &  &  &  & 
\end{tikzcd}
\end{center}

 Let $C_3$ correspond to, \\
\begin{tikzcd}
\sigma_8 &  &  &  &  &  \\
\sigma_5 & \sigma_4 \arrow[r, no head] \arrow[lu, no head] \arrow[l, no head] \arrow[ld, no head] & \sigma_7 \arrow[r, no head] & \sigma_3 \arrow[r, no head] & \sigma_2 \arrow[r, no head] & \sigma_1 \\
\sigma_6 &  &  &  &  & 
\end{tikzcd}
\medskip

 So, $C_3 = [\sigma, \alpha, \phi]$, where
\begin{itemize}
\item $\sigma = (1)(2,3)(4,5)(6,7,8,9)(10)(11)(12,13)(14)$
\item $\alpha = (1,2)(3,4)(5,12)(7,10)(8,11)(9,14)(6,13)$
\end{itemize}
The resolution of $S(1,2)$ is, 

\medskip

\begin{center}
\begin{tikzcd}[row sep=18pt, column sep=15pt]
{(1,2)} &  & {(9,14) \quad (9,14)} \arrow[lldddd] &  & {(7,10) \quad (6,13)} \arrow[lldddd] &  & {\fbox{$(1,2) \quad (3,4)$}} \arrow[lldddd] \\
{\fbox{$(1,2) \quad (3,4)$}} \arrow[u] &  & {(9,14) \quad (8,11)} \arrow[u] &  & {(6,13) \quad (5,12)} \arrow[u] &  &  \\
{(3,4) \quad (5,12)} \arrow[u] &  & {(8,11) \quad (8,11)} \arrow[u] &  & {(5,12) \quad (3,4)} \arrow[u] &  &  \arrow[uu, dotted] \\
{(5,12) \quad (6,13)} \arrow[u] &  & {(8,11) \quad (7,10)} \arrow[u] &  & {(3,4) \quad (1,2) } \arrow[u] &  &  \\
{(6,13) \quad (9,14)} \arrow[u] &  & {(7,10) \quad (7,10)} \arrow[u] &  & {(1,2) \quad (1,2)} \arrow[u] &  & 
\end{tikzcd}
\end{center}

\medskip

\begin{center}
\begin{tikzcd}[row sep=18pt, column sep=15pt]
 & R_8 &  &  &  &  &  &  &  &  &  &  \\
 & R_4 \arrow[rd, "\sigma_4" description] \arrow[u, no head] & \mathfrak{m}_4 & \mathfrak{m}_4 & \mathfrak{m}_4 &  &  &  &  &  &  &  \\
R_5 & R_4 & R_4 \arrow[rd, "\sigma_4" description] \arrow[ll, no head, bend left] & \mathfrak{m}_4 & \mathfrak{m}_4 & R_7 \arrow[rd, "\sigma_7" description] & \mathfrak{m}_7 & R_3 \arrow[rd, "\sigma_3" description] & \mathfrak{m}_3 & R_2 \arrow[rd, "\sigma_2" description] & \mathfrak{m}_2 &  \\
 & R_4 & R_4 & R_4 \arrow[rd, "\sigma_4" description] \arrow[dd, no head, bend right] & \mathfrak{m}_4 & R_7 & R_7 \arrow[ru, no head] & R_3 & R_3 \arrow[ru, no head] & R_2 & R_2 \arrow[r, no head] & R_1 \\
 & R_4 & R_4 & R_4 & R_4 \arrow[ruu, no head] &  &  &  &  &  &  &  \\
 &  &  & R_6 &  &  &  &  &  &  &  & 
\end{tikzcd}
\end{center}
\medskip
\medskip
\medskip
\medskip

Now let $C_4$ correspond to, 

\medskip

\begin{tikzcd}
 &  & \sigma_5 \arrow[d, no head] &  &  &  \\
\sigma_8 & \sigma_7 \arrow[r, no head] \arrow[l, no head] & \sigma_4 \arrow[r, no head] & \sigma_3 \arrow[r, no head] & \sigma_2\arrow[r, no head] & \sigma_1 \\
 &  & \sigma_6 \arrow[u, no head] &  &  & 
\end{tikzcd}

\medskip
\medskip

So, $C_4 = [\sigma, \alpha, \phi]$, where
\begin{itemize}
\item $\sigma = (1)(2,3)(4,5)(6,7,8,9)(10)(11)(12,13)(14)$
\item $\alpha = (1,2)(3,4)(5,6)(7,10)(8,12)(9,11)(13,14)$
\end{itemize}
Then we have the resolution of $S(1,2)$, 

\medskip

\medskip

\begin{center}
\begin{tikzcd}[row sep=18pt, column sep=15pt]
{(1,2)} &  & {(9,11) \quad (8,12)} \arrow[lldddd] &  & {(7,10) \quad (7,10)} \arrow[lldddd] &  & {\fbox{$(1,2) \quad (3,4)$}} \arrow[lldddd] \\
{\fbox{$(1,2) \quad (3,4)$}} \arrow[u] &  & {(8,12) \quad (13,14)} \arrow[u] &  & {(7,10) \quad (5,6)} \arrow[u] &  &  \\
{(3,4) \quad (5,6)} \arrow[u] &  & {(13,14) \quad (13,14)} \arrow[u] &  & {(5,6) \quad (3,4)} \arrow[u] &  &  \arrow[uu, dotted] \\
{(5,6) \quad (9,11)} \arrow[u] &  & {(13,14) \quad (8,12)} \arrow[u] &  & {(3,4) \quad (1,2) } \arrow[u] &  &  \\
{(9,11) \quad (9,11)} \arrow[u] &  & {(8,12) \quad (7,10)} \arrow[u] &  & {(1,2) \quad (1,2)} \arrow[u] &  & 
\end{tikzcd}
\end{center}

\medskip 
\begin{center}
\begin{tikzcd}[row sep=18pt, column sep=15pt]
 &  &  & R_5 &  &  &  &  &  &  &  &  \\
R_8 \arrow[r, no head] & R_7 & \mathfrak{m}_7 & R_4 \arrow[rd, "\sigma_4" description] \arrow[u, no head] & \mathfrak{m}_4 & \mathfrak{m}_4 & \mathfrak{m}_4 &  &  &  &  &  \\
 & R_7 & R_7 & R_4 & R_4 \arrow[rd, "\sigma_4" description] \arrow[ll, no head, bend left] & \mathfrak{m}_4 & \mathfrak{m}_4 & R_3 \arrow[rd, "\sigma_3" description] & \mathfrak{m}_3 & R_2 \arrow[rd, "\sigma_2" description] & \mathfrak{m}_2 &  \\
 &  &  & R_4 & R_4 & R_4 \arrow[rd, "\sigma_4" description] \arrow[dd, no head, bend right] & \mathfrak{m}_4 & R_3 & R_3 \arrow[ru, no head] & R_2 & R_2 \arrow[r, no head] & R_1 \\
 &  &  & R_4 & R_4 & R_4 & R_4 \arrow[ruu, no head] &  &  &  &  &  \\
 &  &  &  &  & R_6 &  &  &  &  &  & 
\end{tikzcd}
\end{center}
\medskip

In all four cases, we have $\cZ(\Lambda) \cong A_0 \cdot k[[z_1,...z_8]]/(z_iz_j)_{i \neq j}$, and 
\[
\cN(\Lambda) \cong \prod_{n=1}^4\left(\overline{k\tilde{\mathbb{A}}_{1}^{eq}}\right) \times \prod_{n=1}^3 \left(\overline{k\tilde{\mathbb{A}}_{2}^{eq}}\right) \times \left(\overline{k\tilde{\mathbb{A}}_{4}^{eq}}\right)
\]

However, it is known that $C_1$ and $C_2$ lie in a quadratic Galois orbit, as do $C_3$ and $C_4$. In other words 
the center and normalization are necessarily isomorphic if two dessins lie in the same orbit, but an 
isomorphism does not imply they are in the same orbit. This is where the pull-back diagrams and the projective 
resolutions come in handy. In particular, Let $A(n) = \overline{k\tilde{\mathbb{A}}_{n}^{eq}}$ be the completion of 
the hereditary algebra. We mentioned in the definition of the pull-back diagrams defining 
dessin order that different choices of $\pi_p$ and $\pi_q$ may lead to nonisomorphic orders. This is where the information given by $\alpha$ lies, in the choices of $\pi_p$ and $\pi_q$ 
\[ 
\xymatrix{
R_{p,q} \ar[r]^{\tilde{\pi}_p}  \ar[d]_{\tilde{\pi}_q} & R_p \ar[d]^{\pi_p} \\
R_q \ar[r]_{\pi_q} & k 
}
\]
which is in general different and non-isomorphic for different choices of $\pi_p$ and $\pi_q$. In the case of completions of path algebras for the above four dessins, the diagrams are of the form
\[
\xymatrix{
\Lambda \ar[rr]^{\pi_2} \ar[d]_{\pi_1} & & \Lambda/\rad(\Lambda)= k^7  \ar[d]^{\pi_q} \\
\cN(\Lambda) \ar[rr]_{\pi_p} & & \cN(\Lambda)/\rad(\cN(\Lambda)) = k^{14}
}
\] 
for all four dessins. 
\end{ex}
\medskip

\begin{ex}
Let us return to the constellation $\mathfrak{c}_4 = [\sigma, \alpha, \varphi]_4$, given by
\[ \sigma = (1,8,3,6)(2,5,4,7), \ \ \alpha = (1,5)(2,6)(3,7)(4,8), \ \ \varphi = (1,2,3,4)(5,6,7,8) \]
\[ \chi(\mathfrak{c}_4)=2-4+2, \quad g(\mathfrak{c}_4) = 1. \]

The symmetry of this dessin is reflected in the projective resolutions. In particular, they are all infinite periodic of period $4$. Let us choose the following labeling, 

\[
\xymatrix{
	1 \ar@/^1pc/[r] \ar@{.>}[r] &   2 \ar@/^1pc/[d] \ar@{.>}[d] \\
	4 \ar@/^1pc/[u] \ar@{.>}[u] &   3 \ar@/^1pc/[l] \ar@{.>}[l]
}
\]
\medskip

Let $S(1,5) = S(1)$ be the simple module corresponding to vertex $1$ in the quiver with relations. Then the projective resolution has the following form
\\
\[
S(1) \leftarrow P(1) \leftarrow \fbox{$P(2) \oplus P(2)$} \leftarrow P(3) \oplus P(3) \leftarrow P(4) \oplus P(4) \leftarrow P(1) \oplus P(1) \leftarrow \fbox{$P(2) \oplus P(2)$} \cdots 
\]
\\
The dessin can be embedded on the torus as follows, \\

\medskip

\begin{center}
\includegraphics[
    page=1,
    width=200pt,
    height=200pt,
    keepaspectratio
]{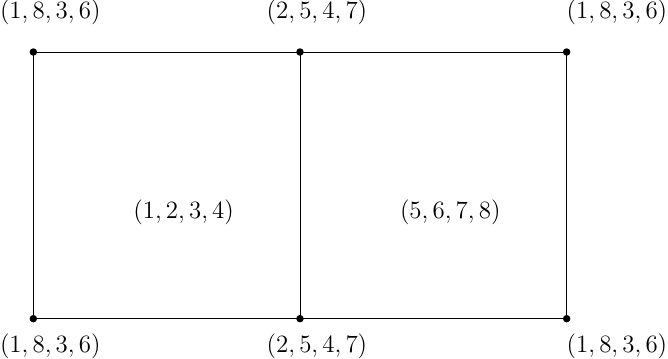}
\end{center}

\medskip

In our shorthand given in terms of $\alpha$ from previous examples, the resolutions looks like, 
\[
S(1,5) \leftarrow (1,5) \leftarrow \fbox{$(2,6) \oplus (2,6)$} \leftarrow  (3,7) \oplus (3,7) \leftarrow (4,8) \oplus  (4,8) \leftarrow (1,5) \oplus (1,5) \leftarrow \fbox{$(2,6) \oplus (2,6)$} \cdots
\]
The pull-back diagram for the dessin order is, 
\[
\xymatrix{
\Lambda \ar[rr] \ar[d] & & \Lambda / \rad(\Lambda)=k^4  \ar[d] \\
    \cN(\Lambda)           \ar[rr] & & \cN(\Lambda)/ \rad(\cN(\Lambda)) = k^8
}
\]
where $\cN(\Lambda) = \left(\overline{k\tilde{\mathbb{A}}_{4}^{eq}}\right) \times \left(\overline{k\tilde{\mathbb{A}}_{4}^{eq}}\right)$. 
\end{ex}

\section{Recollections on Number Fields, $\frp$-adic Completions, and Adeles}

Let $\QQ_p$ be the $p$-\textbf{adic numbers} for some prime $p \in \ZZ$, i.e. all sequences of the form
\[ \sum_{k \in \ZZ} a_k p^k \]
where such sums are bounded below, i.e. there is some $m \in \ZZ$ such that $a_k = 0$ for all $k<m$. The coefficients $a_k$ are from the finite field $\FF_p = \ZZ/p\ZZ$. Let $\ZZ_p$ be the $p$-\textbf{adic integers}. This is a subring of the field $\QQ_p$. Note, $\QQ_p$ is the field of fractions of $\ZZ_p$. In the $p$-\textbf{adic norm} $N_p$, we have 
\[ N_p\left(\sum_{k \in \ZZ} a_k p^k \right) = p^{-m} \]
where $m = \min\{k \in \ZZ: \ a_k \neq 0\}$. Further, 
\[  \ZZ_p = \varprojlim \ZZ/p^k \ZZ \]
Recall, any completion of $\QQ$ is of the form $\QQ_p$, or $\RR$. We then have the \textbf{adeles} (of $\QQ$),
\[ \bA_{\QQ} \cong \left( \widehat{\ZZ} \otimes_{\ZZ} \QQ\right) \times \QQ_{\infty} = \left( \widehat{\ZZ} \otimes_{\ZZ} \QQ\right) \times \RR \]
where $\widehat{\ZZ} = \prod_{p \ \text{prime}} \ZZ_p$ is the \textbf{profinite completion} of the integers, and we have elements of the form $((f_p), f_{\infty}) \in \bA_{\QQ}$, with $(f_p)$ a sequence with finitely many $f_p \notin \ZZ_p$, and $f_{\infty} \in \RR$. We think of $\RR$ as the completion of $\QQ$ with respect to the usual norm which will be denoted by $N_{\infty}$. Now, give $\QQ$ the discrete topology, $\RR$ its usual topology with respect to $N_{\infty}$, and $\widehat{\ZZ}$ the direct product topology. This induces a topology on $\bA_{\QQ}$. There is a diagonal embedding $\QQ \hookrightarrow \bA_{\QQ}$ and 
\[ \QQ \backslash \bA_{\QQ} \cong  \widehat{\ZZ} \times \left(\RR/ \ZZ \right) \]
Note, if we define the ring of \textbf{integral adeles} to be the product
\[ \bA_{\ZZ} = \widehat{\ZZ} \times \RR \]
then we may define the (full) ring of adeles to be
\[ \bA_{\QQ} = \bA_{\ZZ} \otimes_{\ZZ} \QQ.\]

Suppose $K/\QQ$ is a finite field extension and $\sO_K$ its ring of integers. This means $\sO_K$ are elements $\alpha \in K$ such that $f(\alpha) = 0$ for some monic polynomial $f \in \ZZ[x]$. If $K/\QQ$ is Galois, then $K$ is the splitting field of some monic polynomial $f \in \QQ[x]$, and we denote the \textbf{Galois group} by $\cG(K/\QQ)$. 

Any ideal of $\sO_K$ factors uniquely into prime ideals so that that for every prime $p \in \ZZ$ we have
\[ p \sO_K = \prod_{i=1}^r \frp_i \]
If the $\frp_i$ are distinct then the extension $K/\QQ$ is said to be "\textbf{unramified at} $p$", or sometimes "$p$ is \textbf{unramified}". If the primes are not distinct we may write
\[ p \sO_K = \frp_1^{a_1} \frp_2^{a_2} \cdots \frp_r^{a_r} \]
where each $a_i \in \ZZ_{\geq 1}$. If any of the $a_i > 1$, we say $p \in \ZZ$ is \textbf{ramified}, or that "the field extension $K/\QQ$ is \textbf{ramified over} $p$". We call the vector $\mathbf{a}:= (a_1, a_2, ..., a_r) \in \ZZ^r_{\geq 1}$ the \textbf{ramification type} at $p$. To every such $\mathbf{a} \in \ZZ^r_{\geq 1}$ we may order the factors $\frp_i^{a_i}$ so that $a_1 \geq a_2 \geq \cdots \geq a_r$, and we may then associate a \textbf{partition} or a \textbf{Young diagram} to $p$. For the terminology on Young diagrams the texts \cite{F} and \cite{FH} are recommended. 

For a number field $K$, we define the \textbf{adeles} of $K$, denoted $\bA_K$ as follows. First, let 
$\sO_K$ be the completion of the ring of integers in $K$. Let $\mathscr{P}$ be the set of places, i.e. the set of equivalence classes of norms on $\sO_K$. Let $P \subset \mathscr{P}$ be the set of archimedean places. We can identify $N \in \mathscr{P} \backslash P$ with a prime ideal $\frp$ of $\sO_K$, in which case we can define the completion as the product of local rings,
\[ \sO_K = \prod_{\frp \in \mathscr{P} \backslash P} \sO_{\frp}.\]
Let $K_{N}$ for $N \in \mathscr{P}$ be the formal completion of $K$ at the norm $N$. For any $N \in P$ we have $K_N \cong \RR$ or $\CC$. Otherwise $K_N$ is a \textbf{local field} (a generalization of $\QQ_p$) and $\sO_{K,N}$ is open in $K_N$ and compact with all elements $\alpha \in \sO_{K,N}$ satisfying $N(\alpha) \leq 1$. 
Similar to the case where $\QQ$ is the field of fractions of $\ZZ$, i.e. the localization at all nonzero $n \in \ZZ$, we may take $K$ to be the localization of $\sO_K$ at all nonzero elements $x \in \sO_K$. Furthermore, similar to the case of $\QQ_p$ and $\ZZ_p$, we define the integral adeles
\[ \bA_{\sO_K} = \sO_K \times \prod_{N \in P} K_N \]
and the (full) ring of adeles over $K$ is then
\[ \bA_K = \bA_{\sO_K} \otimes_{\sO_K} K.\]

\section{Important Cases of Surface Orders}

\subsection{Algebraic Curves}

\begin{defn}
An \textbf{algebraic curve} $X$ for us will be an irreducible, complete, and nonsingular algebraic variety of dimension one. We will let $k(X)$ denote its field of rational functions, which will be an extension of $k$, of transcendence degree one. Suppose we are given some field extension $K/k$ with generators $\{x_1, ..., x_n\}$, and let $R = k[x_1, ..., x_n] \subset K$ be the associated affine subalgebra of $K$, which corresponds to some some (closed) affine variety $V(k)$ in the affine space $\mathbb{A}_k^n$ over $k$. Let $\overline{V(k)}$ be the closure in $\PP_k^n$. Then $\overline{V(k)}$ is a complete irreducible curve with field of rational functions $K$. If we let $X$ be the noramlization of $\overline{V(k)}$, then $X$ is a complete and nonsingular algebraic curve. 
\end{defn}

In what follows we will look at several very important examples of algebraic curves defined over various fields, and some surface orders that may be associated to them.

\subsection{Algebraic Curves and Loop Algebras $\fg((x))$}
In this section we refer readers to \cite{Fr1} and \cite{Fr2}, take $\FF$ to be the field of fractions of the commutative ring $R$, which is a \textit{complete discrete valuation ring}, and is the ring of integers in $\FF$, and let $\mathfrak{m}$ denote the unique maximal ideal, and $k = R/\mathfrak{m}$ the residue field. Concrete examples of the setup which are important to the theory are
\[
\begin{matrix}
\FF = \CC((x)) & R = \CC[[x]] & \fm = (x) & k = \CC = R/\fm \\
\FF = \hat{\QQ}_{(p)} & R = \hat{\ZZ}_{(p)} & \fm = (p) & k = \ZZ/(p) = \FF_p \\
\FF = \FF_p((x)) & R = \FF_p[[x]] & \fm = (x) & k = \FF_p
\end{matrix} \]

So, for example, let $\FF = \CC((x))$ be the field of formal Laurent series, and $R = \CC[[x]]$ the completion of the polynomial ring $\CC[x]$ (i.e. formal Taylor series) with respect to $\fm = (x)$. Since we work with \emph{formal} power series and \emph{formal} Laurent series, we are allowing infinite sums, and evaluation at $x=0$ in general is the only evaluation allowed. Moreover, define the \textbf{formal punctured disk} to be
\[ D^{\times} = \Spec \CC((x)) \]
and the \textbf{formal disk} to be 
\[ D = \Spec \CC[[x]] \]

Then an order associated to a vertex of a graph cellularly embedded in a Riemann surface, corresponding to a ramification point $x_i$, of degree $n_i$, would be of the form
\[
\Lambda_i = \begin{pmatrix}
\CC[[x_i]] & (x_i) & \cdots & (x_i) \\
\CC[[x_i]] & \CC[[x_i]] & \cdots & (x_i) \\
\vdots & \vdots & \ddots & \vdots \\
\CC[[x_i]] & \CC[[x_i]] & \cdots & \CC[[x_i]]
\end{pmatrix} \ \ \subset \ \ \begin{pmatrix}
\CC[[x_i]] & \CC[[x_i]] & \cdots & \CC[[x_i]] \\
\CC[[x_i]] & \CC[[x_i]] & \cdots & \CC[[x_i]] \\
\vdots & \vdots & \ddots & \vdots \\
\CC[[x_i]] & \CC[[x_i]] & \cdots & \CC[[x_i]]
\end{pmatrix} \ \ \subset \ \
\begin{pmatrix}
\CC((x_i)) & \CC((x_i)) & \cdots & \CC((x_i)) \\
\CC((x_i)) & \CC((x_i)) & \cdots & \CC((x_i)) \\
\vdots & \vdots & \ddots & \vdots \\
\CC((x_i)) & \CC((x_i)) & \cdots & \CC((x_i)) \\
\end{pmatrix}
\]

This is closely related with the theory of reductive algebraic groups over Nonarchimedean fields.\footnote{See for example \cite{Tits}, and the notion of Iwahori subgroups as generalizations of Borel subgroups.}


Now, we have the following diagram
\[
\xymatrix{
\FF = \CC((x_i)) & \\
R = \CC[[x_i]] \ar[u]_{\subseteq} \ar[r]_{\pi}^{f(x_i) \mapsto f(0)} & \CC = \CC[[x_i]]/(x_i)
}\]
which gives the diagram
\[
\xymatrix{
G = \GL_n(\FF)  & \\
K = \GL_n(R) \ar[u]_{\subset} \ar[r]^{\pi}_{f(x) \mapsto f(0)} & \GL_n(\CC) \\
B = \pi^{-1}(B(\CC)) \ar[u]_{\subset} \ar[r]^{\pi}_{f(x) \mapsto f(0)} & B(\CC) \ar[u]_{\subset} 
}\]
where $B_{-}(\CC)$ is the \textbf{Borel subgroup} of \emph{lower triangular matrices} in $\GL_n(\CC)$. The subgroup $B$ is generally called the \textbf{Iwahori subgroup} of $G$. We then have that the \textbf{standard flag variety} can be realized as
\[ \GL_n(\CC)/B(\CC) \]
and the (standard) \textbf{affine flag variety} can be realized as
\[ G/B = \GL_n(\FF)/B \]

\subsection{Algebraic Curves Defined Over $\FF_p((x))$}
Let $X$ be a smooth projective curve over $\FF_p = \ZZ/p\ZZ$ for some prime $p$. Let $\FF = \FF_p(X)$ be the field of rational functions on $X$. Let $x \in X$ be a closed point with maximal ideal $\fm = \fm_x$. Let $\FF_x$ denote the completion of $\FF$ with respect to $\fm$, and let $\sO_x$ be the ring of integers in $\FF_x$. Choose a local coordinate $z_x$ at $x$ given by some rational function vanishing at $x$ with zero of order one. Then we have
\[ \FF_x \cong k((z_x)), \quad \text{and} \ \sO_x \cong k[[z_x]] \]
where $k$ is the residue field of $x$. Let us abuse notation a little and write $k[[z_x]]=k[[x]]$ and $k((z_x)) = k((x))$. Define the two $\sO_x$ orders 
\[ \fb[[x]]_{+} \subset \fg[[x]] \subset \fg((x)) \]
inside the $\FF_x$-matrix algebra $\fg((x)) = \Mat_{n \times n}(\FF_x)$ by, 
\[
\begin{pmatrix}
\sO_x & \fm & \cdots & \fm \\
\sO_x & \sO_x & \cdots & \fm \\
\vdots & \vdots & \ddots & \vdots \\
\sO_x & \sO_x & \cdots & \sO_x
\end{pmatrix}
\subset
\begin{pmatrix}
\sO_x & \sO_x & \cdots & \sO_x \\
\sO_x & \sO_x & \cdots & \sO_x \\
\vdots & \vdots & \ddots & \vdots \\
\sO_x & \sO_x & \cdots & \sO_x
\end{pmatrix}
\subset
\begin{pmatrix}
\FF_x & \FF_x & \cdots & \FF_x \\
\FF_x & \FF_x & \cdots & \FF_x \\
\vdots & \vdots & \ddots & \vdots \\
\FF_x & \FF_x & \cdots & \FF_x
\end{pmatrix}
\]
Suppose that $n = \deg(x)$ is the degree of $x$, i.e. the (finite) number of elements in the field extension $k/\FF_p$. Let $\cG_x = \cG(k/\FF_p)$ be the Galois group of the field extension. 

Let $G((x)) = \GL_n(\FF_p((x)))$ and let $\fg((x)) = \gl_n((x)) = \Mat_{n \times n}(\FF_p((x)))$ be its Lie algebra. Let 
\[ K_N = \{g \in G((x)): \ g \equiv 1 \ \mod \ x^N \} \]
These are often called \textbf{congruence subgroups} and define a base of open neighborhoods of $\id_n \in \GL_n(\FF((x)))$ giving the topology for $\GL_n(\FF((x)))$. Now, notice, the algebra $\fb[[x]]_{+}$ is a \emph{hereditary order}, which can be associated to a point $x \in X$. Suppose we have a finite list of points $\{x_1, ..., x_m\} \subset X$ such that $\deg(x_i) = n_i \geq 1$. Suppose further we have a constellation $C = [\sigma, \alpha, \phi]$ with respect to the set $\{x_1, ..., x_m\}$. This naturally defines a cellularly embedded graph, and we may take $\fb[[x_i]]_{+}$ to be the hereditary orders associated to the vertices of the cellularly embedded graph $\Gamma$. Combinatorially, this is sufficient information to define a surface order over $X$. The gluing will be given by $\alpha$ as usual, and the cyclic order around each vertex will be given by each cycle $\{\sigma_1, ..., \sigma_m\}$ and the order of each $\sigma_i$ will be $\deg(x_i)$. 

\subsection{Smooth Projective Curves Defined Over Arbitrary Fields $k$}
Suppose $X$ is a smooth projective curve defined over any field $k$. Let $F = k(X)$ be the function field of $X$. We may define the ring of adeles for $F$, and we may define a surface order as follows. Let $\cG = \cG(\overline{F}/F)$ be the Galois group, which may be interpreted as the group of deck transformations of the maximal ramified cover of $X$. Let $x \in X$ be a closed point of $X$ with maximal ideal $\fm = \fm_x$. Let $k_x$ be the residue field at $x$. Let $\cO_X$ be the structure sheaf of $X$ so that for open $U \subset X$, $\cO_X(U)$ is the sheaf of rings over $U$, and $\cO_X(U)_x = \cO_{X,x}$ is the local ring with respect to the maximal ideal $\fm_x$ corresponding to $x$. Let $\widehat{\cO}_{X,x}$ be the completion of the local ring, and let $K_x$ be the local field, i.e. the field of fractions of $\widehat{\cO}_{X,x}$. For a finite number of closed points $\{x_1, ..., x_n\} \subset X$, we may associate a \textbf{local surface order} $\Lambda_i$ defined as, 
\[ \Lambda_i = \begin{pmatrix}
\cO_X(U)_{x_i} & \fm_{x_i} & \cdots & \fm_{x_i} \\
\cO_X(U)_{x_i} & \cO_X(U)_{x_i} & \cdots & \fm_{x_i} \\
\vdots & \vdots & \ddots & \vdots  \\
\cO_X(U)_{x_i} & \cO_X(U)_{x_i} & \cdots & \cO_X(U)_{x_i}
\end{pmatrix} \ \subset \ \Mat_{n_i \times n_i}(\cO_X(U)_{x_i}) \ \subset \ \Mat_{n_i \times n_i}(K_x) .\]

For any affine ring $A = k[x_1, ..., x_n]$ over a field of characteristic $p>0$, we may define the \textbf{Frobenius map} which replaces coefficients of $f \in A$ by $a \mapsto a^p$. Then for any (closed) affine scheme $X$ over $\mathbb{A}_k$ given by polynomials $f_1, f_2, ...$ we may define a map
\[ X(k) \to X^{(p)}(k) \]
given by $(a_1, ..., a_n) \mapsto (a_1^p, ..., a_n^p) \in \mathbb{A}_k^n$. This can be realized as a map of affine $k$-schemes. Let $A = \cO_X(X) = k[x_1, ..., x_n]/(f_1, f_2, ...)$, and define
\[ A^{(p)} = \cO_X(X^{(p)}) = k \otimes_{k,\sigma} A = k[x_1, ..., x_n]/(f_1^p, f_2^p, ...).\]
Then we let $F_X:X \to X^{(p)}$ be given by $\Spec$ of the following $k$-algebra morphism
\[ c \otimes a \mapsto ca^p: \ \ k \otimes_{k,\sigma} A \to A.\]
This may be extended to arbitrary schemes in the usual way by applying the definition to affine covers. For a field $k$ of characteristic $p>0$, and for each section of the structure sheaf $\cO_X(U_i)$, where $U_i$ is an open neighborhood of $x_i$ not containing the other ramification points $x_j \neq x_i$ (or more generally for any section $\cO_X(U)$), we may define the \textbf{absolute Frobenius morphism}
\[ \sigma_X: X \to X \]
which acts as the identity on the underlying set $|X|$ of $X$, and which sends 
\[ f \mapsto f^{(p)}: \ \ \cO_X(U) \mapsto \cO_X(U).\]
This may be represented by the matrix
\[
\sigma_i = \begin{pmatrix}
0 & 0 & 0 & \cdots & 0 & (x_i) \\
1 & 0 & 0 & \cdots & 0 & 0 \\
0 & 1 & 0 & \cdots & 0 & 0 \\
\vdots & \vdots & \vdots & \ddots & \vdots & \vdots \\
0 & 0 & 0 & \cdots & 0 & 0 \\
0 & 0 & 0 & \cdots & 1 & 0 \\
\end{pmatrix}
\]
acting on the local order (over $\cO_{X,x}$). For more on this morphism see \cite{M1}. 

We may define a \textbf{relative Frobenius morphism} $F_X: X \mapsto X^{(p)}$ as in \cite{M1} by the diagram

\begin{center}
\medskip
\begin{tikzcd}
X \arrow[rd, "F_X", dashed] \arrow[rrrd, "\sigma_X", bend left] \arrow[rdd, bend right] &  &  &  \\
 & X^{(p)} \arrow[rr] \arrow[d] &  & X \arrow[d] \\
 & \mathbf{Spm}(k) \arrow[rr, "\sigma_{\mathbf{Spm}(k)}"'] &  & \mathbf{Spm}(k)
\end{tikzcd}
\medskip
\end{center}
Moreover, we may define $F^n_X$ simply by iterating
\medskip
\begin{center}
\begin{tikzcd}
X \arrow[r, "F"] & X^{(p)} \arrow[r, "F"] & \cdots \arrow[r, "F"] & X^{(p^n)}
\end{tikzcd}
\end{center}
\medskip
which amounts to replacing the maximal ideal $\fm$ with $\fm^n$ in the matrix $\sigma_i$ in the local orders (or vertex orders for an embedded graph). 
If we take the completions denoted by $\widehat{\sO}_i = \widehat{\cO_X}(U)_{x_i}$, with local fields $K_i = K_{x_i}$, and maximal ideals $\fm_i = \fm_{x_i}$, we may define the orders
\[
\begin{pmatrix}
\widehat{\sO}_i & \fm_i & \cdots & \fm_i \\
\widehat{\sO}_i & \widehat{\sO}_i & \cdots & \fm_i \\
\vdots & \vdots & \ddots & \vdots \\
\widehat{\sO}_i & \widehat{\sO}_i & \cdots & \widehat{\sO}_i \\
\end{pmatrix} \ \subset \ \Mat_{n_i \times n_i}(\widehat{\sO}_i) \ \subset \ \Mat_{n_i \times n_i}(K_i).
\]
In this case we may again define a \textit{Frobenius morphism} given by the same matrix $\sigma_i$. The Frobenius morphism, as described in \cite{M1} (\S 2.d) is \emph{functorial, compatible with products, and commutes with base change}. 
\subsection{Nonabelian Class Field Theory and Number Fields}
Suppose $K/\QQ$ is a number field given by a finite extension of $\QQ$ ramified over $p$. Let $\sO_K$ be the ring of integers, $\mathscr{P}$ the set of places, and $P \subset \mathscr{P}$ the archimedean places. Let $K_{\frp}$ be the formal completion at the place $\frp \in \mathscr{P} \backslash P$, and $K_x$ be the completion for archimedean places $x \in P$. Let $\frp \notin P$ so that the local field $K_{\frp}$ has corresponding local ring $\sO_{\frp}$ for some prime $\frp$. Let $\bA_{\sO_K}$ be the integral adeles over $K$ and let $\bA_K$ be the (full) ring adeles. For the unramified $p \in \ZZ$, let $p \sO_K = \frp_1^{a_1} \frp_2^{a_2} \cdots \frp_r^{a_r}$, $a_i = 1$. Define
\[ \Lambda_i = \begin{pmatrix}
\sO_{\frp_i} & \frp_i & \cdots & \frp_i \\
\sO_{\frp_i} & \sO_{\frp_i} & \cdots & \frp_i \\
\vdots & \vdots & \ddots & \vdots \\
\sO_{\frp_i} & \sO_{\frp_i} & \cdots & \sO_{\frp_i}
\end{pmatrix} \ \subset \ \Mat_{a_i \times a_i}(\sO_{\frp_i}) \ \subset \ \Mat_{a_i \times a_i}(K_{\frp_i}).\]
Now, the Galois group $\cG(K/\QQ)$ permutes the factors of $p \sO_K$. Define the usual automorphism $\sigma_i$ by 
\[ \sigma_i = \begin{pmatrix}
0 & 0 & 0 & \cdots & 0 & \frp_i \\
1 & 0 & 0 & \cdots & 0 & 0 \\
0 & 1 & 0 & \cdots & 0 & 0 \\
\vdots & \vdots & \vdots & \ddots & \vdots & \vdots \\
0 & 0 & 0 & \cdots & 0 & 0 \\
0 & 0 & 0 & \cdots & 1 & 0 \\
\end{pmatrix} \]
From the beginning of this section, we may associate to the field extension $K/\QQ$ an algebraic curve $X$. In particular, the affine subalgebra, given by the extension $\QQ[x_1, ..., x_n] \subset K$, will correspond to some closed affine variety in $\QQ^n$, which has closure $X \subset \PP_{\QQ}^n$, a nonsingular, complete algebraic curve. 

Now, we wish to treat \emph{ramified primes}. We will follow Serre's \cite{S2} closely for the terminology and notation. Let $\la g \ra =C_N$ be a cyclic group of order $N$. Then we may 
identify $C_N$ with $\ZZ/ N \ZZ$, or with the multiplicative version $\la e^{\zeta_i(N)}\ra$, where $\zeta_i(N)$ is a 
primitive $N^{th}$ root of unity (which can be identified with some element in $(\ZZ/N\ZZ)^{\times}$, for example 
the Frobenius automorphism). We make the identification and let $C_N = \la e^{\zeta(N)} \ra$, with $\zeta(N)$ 
the first primitive $N^{th}$ root of unity. Now, let $P \subset \ZZ$ be the set of primes, and let $K/\QQ$ be a 
finite Galois extension with ring of integers $\cO = \cO_K$. Now, suppose that $\cG = \cG(K:\QQ)$, the Galois 
group, is cyclic. Then, it is a well known fact (Artin's Reciprocity Law for Abelian Class Field Theory), that for any 
group homomorphism $\rho: \cG \to \CC^{\times}$, there exists some $N_{\rho} \in \ZZ_{\geq 0}$ and some 
Dirichlet character $\chi_{\rho}$, such that 
\[ \chi_{\rho}: (\ZZ/N_{\rho} \ZZ)^{\times} \to \CC^{\times}, \quad \rho(\mathfrak{F}(p)) = \chi_{\rho}(p) \]
for all primes $\frp /p$ \emph{unramified} in $K$. Now, suppose $\{p_1, p_2, ..., p_r\}$ is the finite set of 
ramified primes. In particular, we have 
\[ p_i\cO = \frp_{i,1}^{e(i,1)}\frp_{i,2}^{e(i,2)}, \cdots, \frp_{i,g(i)}^{e(i, g(i)} \]
where $g(i)$ denotes the number of primes $\frp_{i,j} \bigg| p_i$. Moreover, we have that $\cO/\frp_{i,j}$ is a 
finite extension of $\ZZ/p\ZZ = \FF_p$, and is therefore isomorphic to $\ZZ/p^{f(i)}\ZZ = \FF_q$. We will call 
$e(i,j)$ the \textbf{ramification degree} of $\frp_{i,j}$, and $f(i)$ the \textbf{residue degree}. Then we have $f_i = 
[\cO/\frp_{i,j}: \ \ZZ/p_i\ZZ]$. Now, fix $p_i$ and observe, the degree of the field extension $n = [K: \ \QQ]$ is 
also the degree of the $\ZZ/p_i\ZZ$ algebra 
\[ \cO/p_i\cO \cong \prod_{\frp_{i,j} |\  p_i} \cO/\frp^{e(i,j)} \]
and we have 
\[ n = \sum_{\frp_{i,j} | \  p_i} e(i,j)f(i). \]
Moreover, we also have
\[ n = e(i,j)f(i)g(i) \]

Now, suppose that $D(\frp_{i,j}) \subseteq \cG(K/\QQ)$ is a subgroup fixing the prime ideal $\frp_{i,j} \bigg| p_i$. 
We call it the \textbf{decomposition subgroup of} $\frp_{i,j}$. Moreover, we have that $D(\frp_{i,l} \bigg| p_i$ is 
conjugate in $\cG$ to $D(\frp_{i,j})$. The index $g(i) = [D(\frp_{i,j}): \ \cG]$ is the number of primes over $p_i$ 
inside $\cO$. Such a subgroup corresponds to a field extension $K_i/\QQ$, with $K \supset K_i$. Moreover
\[ g(i) = [K_i:\ \QQ], \quad [K:\ K_i] = e(i,j)f(i), \quad \cG(K/K_i) = D(\frp_{i,j}).\]
There is a homomorphism
\[ \phi: D(\frp_{i,j}) \to \cG(\FF_q/\FF_p) \]
with $\ker(\phi) = I_{i,j}$ the \textbf{inertia group} of $\frp_{i,j}$. In general we have an isomorphism
\[ D(\frp_{i,j})/I_{i,j} \cong \cG(\FF_q/\FF_p) \]
for the residue extension $\FF_q/\FF_p$. Now, we know that $\cG(\FF_{q_i}/\FF_{p_i}) \cong C_N$ is cyclic of 
order $N = f_i$, where $p_i = q^{f_i}$ is the residue degree. If $I_{i,j}$ is nontrivial, it must be a cyclic subgroup 
of order $g(i)$ since it is the isotropy group of the $\frp_{i,j}$. In particular setting $D_{i,j} = D(\frp_{i,j})$, we 
have that $[I_{i,j}:\ D_{i,j} ] = f_0(i)$, and $f(i) = f_0(i)p_i^{g(i)}$. Now, identifying a generator of $C_N$ with the 
matrix
\[ \mathfrak{F}(p_i) = \begin{pmatrix}
0 & 0 & 0 & \cdots & 0 & p_i \\
1 & 0 & 0 & \cdots & 0 & 0\\
0 & 1 & 0 & \cdots & 0 & 0 \\
\vdots & \vdots & \vdots & \ddots & \vdots & \vdots \\
0 & 0 & 0 & \cdots & 0 & 0 \\
0 & 0 & 0 & \cdots & 1 & 0 \\
\end{pmatrix} \]
realized as an element of $\Mat_{f_0(i) \times f_0(i)}(\ZZ/p_i^{f(i)}\ZZ)$. This element as an element of $
\Mat_{f_0(i) \times f_0(i)}(\ZZ)$ generates an infinite cyclic group. Notice that $\mathfrak{F}(p_i)^{f_0(i)}$ is
\[ \mathfrak{F}(p_i) = \begin{pmatrix}
0 & 0 & 0 & \cdots & 0 & p_i^2 \\
p_i & 0 & 0 & \cdots & 0 & 0\\
0 & p_i & 0 & \cdots & 0 & 0 \\
\vdots & \vdots & \vdots & \ddots & \vdots & \vdots \\
0 & 0 & 0 & \cdots & 0 & 0 \\
0 & 0 & 0 & \cdots & p_i & 0 \\
\end{pmatrix} \]
So that there are infinite cyclic subgroups generated by powers of $\mathfrak{F}(p_i)$ of index 
\[ [\mathfrak{F}(p_i)^{f_0(i) \cdot n}:\ \mathfrak{F}(p_i)] = n \]
We may think of this as corresponding to the profinite completions 
\[ \varprojlim \ZZ/p_i^n \ZZ \]
with the corresponding inclusions of cyclic matrix groups. In this setup, there is an infinite cyclic group 
corresponding to the decomposition group, inertia group, and Galois group. One of the major successes of abelian class field theory is obtaining a result on primes lying in arithmetic progressions and splitting behaviors. Having a theory which allows one to understand the nonabelian generalization in terms of the arithmetic of $\QQ$ along is of course important. So, once the above construction is carried out for each ramified prime, we may realize this as a surface order via a gluing given by the action of the Galois group. Furthermore, this provides a way of explicitly construction Artin representations occuring "\emph{in nature}", rather than attempting to understand such representations in terms of character theory alone.\footnote{See \cite{S2} Chapter VI, \cite{S1} Part III Chapter 19).} The theory of Brauer graph algebras is also implicitely contained in the theory of surface algebras, as all Brauer graph algebras are quotients of surface algebras. All finite group algebras are of this form. Moreover, the modular representation theory studied via the methods of $R$-orders is naturally contained in the theory of surface orders as we have constructed them.\footnote{See for example \cite{CR1, CR2} and \cite{B1, B2}.} So we have a way of tying these theories together in order to develop a rather complete \emph{nonabelian} class field theory. One very convenient aspect of this is that it is entirely in terms of "classical" or "well established" mathematical theories, which have all been studied for decades now, and which now may be applied via surface algebras and surface orders to obtain a very cohesive and intuitive picture of the Arithmetic Langlands Program, as well as the Geometric Langlands Program, all in a single "unified" theory.

\section{Artin's L-functions}

\subsection{Torus Actions, Characters,  and Weights of Reductive Algebraic Groups}
For an introduction to semigroup rings, toric varieties, and the related Geometric Invariant Theory quotients (GIT quotients) the reader is referred to \cite{MS} \S 7 and \S 10. For a brief explanation of semi-invariants see \cite{H1} \S 11.4. For a more theoretical background see \cite{SV}, \cite{DW}, \cite{D}, and \cite{Do1, Do2, Do3, Do4}. For the standard information on algebraic groups we use \cite{M1} and \cite{H1}

\begin{defn}
A \textbf{character}, $\chi$, of an algebraic group $G$ over a field $k$ is a homomorphism
\[ G \to \mathbb{G}_m \]
where $\mathbb{G}_m$ is the \textbf{multiplicative group} in the center of $\GL(V)$, represented by $ \cO(\mathbb{G}_m) = k[t,t^{-1}] \subset k(t)$. Any character defines a representation of $G$ on a vector space $V$ by defining eigenspaces for the action $\rho(g) \cdot v = \chi(g)v$. This gives
\[ \xymatrix{ G \ar[r]^{\chi} & \mathbb{G}_m \ar[r] & \GL(V)} \]
via the map
\[ g \mapsto \begin{pmatrix}
\chi(g) & & 0 \\
& \ddots & \\
0 & & \chi(g)
\end{pmatrix}.\]
We can define such actions on subspaces $W \subset V$ if $W$ is stable under the $G$-action. We can define the \textbf{product} of two characters via $(\chi_1 \chi_2)(g) \cdot v = \chi_1(g)\chi_2(g)v$. So the set of all characters $\mathfrak{X}(G)$ is a commutative group. Let 
\[ V_{\chi} = \{v \in V: \ \rho(g) \cdot v = \chi(g)v \ \forall \ g \in G\} \]
be the $G$-stable subspace of \textbf{semi-invariants} of \textbf{weight} $\chi$. 
\end{defn}

Note, any representation
\[ \rho: G \to \GL(V) \]
induces a map of characters given by the commutative diagram,
\[ \xymatrix{
\mathfrak{X}(\rho(G)) \ar[r]  \ar[dr] & \mathfrak{X}(\GL(V)) \ar[d] \\
&  \mathfrak{X}(G) 
}\]
with $\mathfrak{X}(\rho(G)) \hookrightarrow \mathfrak{X}(G)$ being injective. If the representation $\rho$ is \emph{faithful} (i.e. injective) then all characters of the group $G$ may be identified with characters of $\rho(G)$. Generally speaking, given a linear representation $(V, \rho)$ of $G$, characters are obtained via taking the trace of $\rho(g) \in \GL(V)$.

\subsection{Dimension Vectors, and (Semi)Invariant Polynomials}

Let $\Lambda$ be a surface algebra for some Riemann surface $X$. Then $\Lambda$ is given by a quiver with relations $kQ/I$. A \textbf{representation of a quiver}, "$V$" is simply an assignment of a vector space $V(x)$ to each vertex $x \in Q_0$, and a linear map $V(a): V(x) \to V(y)$ for each arrow $a \in Q_1$ with \textbf{tail} $ta=x \in Q_0$ and \textbf{head} $ha=y \in Q_0$. 

Now, let $\cG$ be the Galois group of a finite extension of number fields $K/F$ and let 
\[ \{\chi_1, \chi_2, ..., \chi_r\} \]
be the complete set of irreducible characters of $\cG$. Let $\Lambda_i$ be the local order for a ramified prime $p \in \cO_F$, with $p \cO_K = \frp_1(p) \frp_2(p) \cdots \frp_{n_1}(p)$. This can be identified with a cyclic quiver $Q(p)$ with $n_i$ arrows. Given a finite dimensional representation of $Q(p)$, at each vertex $x \in Q(p)$ we have a finite dimensional vector space $V(x) \cong K^{d_j}$ with $d_j \in \ZZ_{\geq 0}$ and an action of $\GL(V(x)) \cong \GL_K(d_j)$ via base change. Now, the reductive linear algebraic group

\[ \GL(\bd) = \prod_{j=1}^{n_i} \GL_K(d_j) \]
acts on the representation $V(p)$ of $Q(p)$. Now, under this action, an orbit corresponds to an isomorphism class of representations of $Q(p)$ with fixed dimensions $\bd(x) = \dim_KV(x)$. Let $\bd = (d_1, d_2, ..., d_{n_i})$ be the \textbf{dimension vector} for the space of all representations with these dimensions. This space is identified with the affine space
\[ \rep_{Q(p)}(\bd) = \bigoplus_{a \in Q(p)_1} \Hom_K(V(ta), V(ha)). \]
The coordinate ring of this space will be denoted $R(\bd) = K[\rep_{Q(p)}(\bd)]$. We obtain an action of $\GL(\bd)$ on this space by $g \cdot f(x) = f(g^{-1}x)$. It is well known (see for example \cite{LP}, \cite{L}, \cite{D}, \cite{SV}) that the \textbf{ring of polynomial invariants}, denote $R(\bd)^{\GL(\bd)}$, are given by traces corresponding to the "simple oriented cycles" up to cyclic permutation. There is one generator for each such cycle of the noncommutative normalization of the surface algebra $\Lambda$ (i.e. up to cyclic permutation each $Q(p)$ contributes a single generator to the space invariants). So we have $n_i$ generators of the ring of invariants. Each is given by a trace of a square matrix $M_x = V(a_{n_i})V(a_{n_i-1}) \cdots V(a_2)V(a_1)$, where $ta_1 = ha_{n_i}$ for some fixed labeling of the arrows. We may define a representation of the cyclic group $\ZZ/n_i\ZZ$ on the representation $V(p)$ of $Q(p)$ by its action on the local order 
\[ \Lambda_i = \begin{pmatrix}
\cO_{K, \frp} & \frp & \cdots & \frp \\
\cO_{K, \frp} & \cO_{K, \frp} & \cdots & \frp \\
\vdots & \vdots & \ddots & \vdots \\
\cO_{K, \frp} & \cO_{K, \frp} & \cdots & \cO_{K, \frp} \\
\end{pmatrix} \subset \Mat_{n_i \times n_i}(\cO_{K,p}) \subset \Mat_{n_i \times n_i}(K_{\frp})\]
given by the matrix
\[ \sigma_i = \begin{pmatrix}
0 & 0 & 0 & \cdots & 0 & \frp_i \\
1 & 0 & 0 & \cdots & 0 & 0 \\
0 & 1 & 0 & \cdots & 0 & 0 \\
\vdots & \vdots & \vdots & \ddots & \vdots & \vdots \\
0 & 0 & 0 & \cdots & 0 & 0 \\
0 & 0 & 0 & \cdots & 1 & 0 \\
\end{pmatrix} \]
Recall, the path algebra of the cyclic quiver $Q(p)$ is completed to $\Lambda_i$ with respect to its arrow ideal. Now, for every ramified prime we have a way of defining a representation of a cyclic subgroup of $\cG(K/F)$. In particular, if we let 
\[ \Delta_{\bd}: \GL_K(d_1) \times \GL_K(d_2) \times \cdots \times \GL_K(d_{n_1}) \hookrightarrow \GL_K(\sum_j d_j) \]
be the diagonal embedding. In more compact notation
\[ \prod_{j=1}^{n_i} \GL_K(d_j) = \GL_K(d_1) \times \GL_K(d_2) \times \cdots \times \GL_K(d_{n_1}).\]
Now, observe the following diagram
\[ \xymatrix{
& & &\GL_K(\sum_j d_j) \ar@/_/[dd]_{\mathbf{res}} \\
\ZZ/n_i\ZZ \ar[r]^{\sigma_i}  & S_{n_i}  \ar[r]^{\rho(\sigma_i)} & \GL_K(n_i) \ar[ur]^{\rho(\sum_j d_j)} \ar[dr]_{\rho(d_1) \times \cdots \times \rho(d_{n_1}) \quad }  & \\
& & & \prod_{j=1}^{n_i} \GL_K(d_j)  \ar@/_/[uu]_{\mathbf{ind}}
}\]
be a representations of the cyclic group. We can first define a one dimensional character $\chi(g) = \omega$, where $\la g \ra = \ZZ/n_i\ZZ$ and $\omega$ is an $n_i^{th}$ root of unity. We may the induce to $S_{n_i}$ which embeds in $\GL_K(n_i)$ as a permutation matrix. \footnote{We can think of the infinite cyclic group $C_{\infty}$ as acting on the unit circle $\{e^{i \theta}\} \subset \CC$ by sending the generator $z \mapsto e^{i \theta}$ for some irrational $\theta \in [0, 2\pi)$. This action can be identified with a "\emph{non-commutative torus}", which is a noncommutative analogue of elliptic curves (see \cite{Ma}). This seems to have some very interesting implications, and suggests thinking of surface algebras and surface orders as gluings of non-commutative elliptic curves.}

Now, weights $\chi: \prod_{j=1}^{n_i} \GL_K(d_j) \to K$ of the form
\[ \chi(g_1, g_2, ..., g_{n_i}) = \prod_{j=1}^{n_i} \det(g_j)^{\chi_j} \]
correspond to dimension vectors of representation of the quiver $Q(p)$ (really for any quiver) as explained in \cite{Sc1} and \cite{SV}. In fact, one can obtain all such weights as maps between projective representations as explained in \cite{SV} and \cite{DW}. All polynomial semi-invariants (and thus invariants) arise in this way. Moreover, all rational invariants are obtained as quotients of polynomial semi-invariants, and can therefor be obtained in this way. This gives us a direct connection between the characters of the Galois group (or any automorphism group of a finite extension of number fields) and the determinants defining the Artin L-functions. Let us now recall some information about Artin L-functions. 

\subsection{Artin L-Functions}

\emph{Before beginning this section, we must note one crucial subtlety which is not ellaborated on here, but which can be found in the literature \cite{LP}, \cite{L}, \cite{Do1, Do2, Do3, Do4} and various others. The algebraic group action $\GL(\bd) = \prod_{j=1}^{n_i} \GL(d_j)$ on a quiver representation given by a dimension vector $\bd = (d_1, ..., d_{n_i})$ can be embedded diagonally into a larger $\GL(n)$. Further, under conjugation in the larger $\GL(n)$ there is a bundle structure on the representations of the quiver $Q$ given by different dimension vectors with the sum $\sum_j d_j = n$ adding to up to the same $n$. This larger $\GL(n)$ is important in the theory and applications of our surface algebras and surface orders, but setting all of this technology up would require a lengthy exposition. It is extremely important to note though, that this setup allows one to consider non-Galois extensions $K/F$, where the exponent of factors over a ramified prime of the extension are not constant. In particular, this allows us to understand and prove results for Artin L-functions and Langlands Correspondence both in the classical and geometric sense for non-Galois extensions as well as the nicer Galois extensions.}

Let us recall some information on Artin's L-functions. First, let $K/\QQ$ be a finite Galois extension, $\sO_K$ the ring of integers, and $p \in \ZZ$ a ramified prime. Let 
\[\frp_1^{a_1} \frp_2^{a_2} \cdots \frp_{n_i}^{a_{n_i}}\]
be the primes lying over $p$. Since $K/\QQ$ is Galois we have
\[ a_1 = a_2 = \cdots = a_{n_i}.\]
Choose any of the primes $\frp_i$ lying over $p$ (since the choice is well defined up to conjugation and determinants and traces are invariant under base change by $\GL_n$). We have the following data

\begin{enumerate}
\item $\mathfrak{D}_i = \mathfrak{D}_{\frp_i}$, the \textbf{decomposition group}, defined as
\[ \mathfrak{D}_i := \{g \in \cG(K/\QQ): \ g(\frp_i) = \frp_i \}.\]
\item $\mathfrak{I}_i$, the \textbf{Inertia subgroup}, define as
\[ \mathfrak{I}_i := \{ g \in \mathfrak{D}_i: \ g(x) = x (\text{mod } \frp_i) \ \forall \ x \in \sO_K \}.\]
\item The \textbf{Frobenius automorphism} $\sigma_i = \sigma_{\frp_i}$ which generates the cyclic group
\[ \mathfrak{D}_i/\mathfrak{I}_i \cong \cG\left(\sO_K/\frp_i \bigg/ \ZZ/p\ZZ\right),\]
where $\sO_K/\frp_i \cong \FF_{p^{n_i}} = \ZZ/p^{n_i}\ZZ$. 
\end{enumerate}

Given a linear representation of the Galois group,
\[ \rho: \cG(K/\QQ) \hookrightarrow \GL(V) \]
let $V^{\mathfrak{I}_i} \subset V$ be the invariant subspace under the action of $\mathfrak{I}_i$. Then $\rho(\sigma_i)$ is well define on $V^{\mathfrak{I}_i}$. Define an \textbf{Euler factor} by
\[ L_p(\rho, s):= det\left(I-\rho(\sigma_i)N\frp_i^{-s}\bigg|_{V^{\mathfrak{I}_i}}\right).\]
Again, this depends only on the prime $p \in \ZZ$ and not on the choice of prime $\frp_i$ over $p$. Now, define an \textbf{Artin L-function} associated to the representation $\rho: \cG(K/\QQ) \hookrightarrow \GL(V)$ by
\[ L(\rho, s) = L_{K/\QQ}(\rho, s) = \prod_{p \in \ZZ \ \text{prime}} L_p(\rho, s)^{-1}.\]
Now, at \textbf{unramified primes} the definition simplifies since $\mathfrak{I}_{\frp} = \{1\}$ and $V^{\mathfrak{I}_{\frp}} = V$. 

What we want to do now is give an explicit description of representations of $\cG(K/\QQ)$ and the Frobenius automorphisms $\rho(\sigma_{\frp})$ for a prime $\frp$ in $\sO_K$ lying over $p \in \ZZ$. Let $\widehat{\sO}_{\frp} = \widehat{\sO}_{K, \frp}$ denote the completion at the prime $\frp$. Similarly, let $\ZZ_p$ denote the ring of $p$-adic integers. Let $K_{\frp}$ denote the corresponding local field, and $\QQ_p$ the $p$-adic numbers. 

For \emph{any} prime $p \in \ZZ$, ramified or not, and for $\frp$ a prime in $\sO_K$ over $p$, define the local orders
\[
\begin{pmatrix}
\widehat{\sO}_{\frp} & \frp & \cdots & \frp \\
\widehat{\sO}_{\frp} & \widehat{\sO}_{\frp} & \cdots & \frp \\
\vdots & \vdots & \ddots & \vdots \\
\widehat{\sO}_{\frp} & \widehat{\sO}_{\frp} & \cdots & \widehat{\sO}_{\frp} 
\end{pmatrix} \subset \Mat_{n \times n}\left(\widehat{\sO}_{\frp}\right) \subset \Mat_{n \times n}\left(K_{\frp}\right)
\]
and
\[
\begin{pmatrix}
\ZZ_p & p & \cdots & p \\
\ZZ_p & \ZZ_p & \cdots & p \\
\vdots & \vdots & \ddots & \vdots \\
\ZZ_p & \ZZ_p & \cdots & \ZZ_p \\
\end{pmatrix} \subset \Mat_{n \times n}\left(\ZZ_p \right) \subset \Mat_{n \times }\left(\QQ_p\right).
\]
Now, for ramified primes $p_j \in \ZZ$, we let $n_j$ be the exponent of a prime $\frp_j$ in the factorization of 
$p \sO_K$, i.e. the order of the ramification over $p$, which remember again, is constant for the primes over a fixed $p \in \ZZ$ since $K/\QQ$ is Galois. We will treat the general case as well, but for now for simplicity we stay with the assumption that the finite extension is Galois. Now, the set of primes $\{p_j\}_{j=1}^{r}$ which are ramified is finite. Defining the hereditary orders over the ramified primes as above inside an $n_i \times n_i$ matrix algebra, and taking the product of all of these, we may embed the product diagonally into a larger matrix algebra of size $n = \sum_{j=1}^r n_i$. We may think of these larger matrix algebras as
\[ \Mat_{n \times n}\left(\widehat{\sO}_K \right) \subset \Mat_{n \times n}\left(\bA_{K} \right), \]
and 
\[ \Mat_{n \times n}\left(\widehat{\ZZ}\right) \subset \Mat_{n \times n}\left(\bA_{\QQ} \right) \]
where $\bA_{\QQ}$ are the adeles of $\QQ$ and $\bA_{K}$ are the adeles of $K$. Now, the characteristic polynomials
\[ \det\left( I - \rho(\sigma_i)N\frp_i^{-s}\right) \]
for the matrix algebras corresponding to the local orders over a \textbf{ramified prime} can be computed. For the Galois group $\cG(K/\QQ)$ we have irreducible characters $\chi_1, \chi_2, ..., \chi_r$ corresponding to the irreducible representations of $\cG(K/\QQ)$. moreover, we have

\begin{theorem}
(Artin's Theorem, see \cite{S1} pg. 70): Let $G$ be a family of subgroups of the finite group $\cG$. Let 
\[ \mathbf{Ind}: \bigoplus_{H \in G} R(H) \to R(\cG) \]
be the homomorphism defined by the family of $\mathbf{Ind}_H^{\cG}$. Then the following are equivalent:
\begin{enumerate}
\item $\cG$ is the union of the conjugates of the subgroups belonging to $G$.
\item The cokernel of $\mathbb{Ind}: \bigoplus_{H \in G} R(H) \to R(\cG)$ is finite.
\item For each character $\chi$ of $\cG$, there exist virtual characters $\chi_H \in R(H)$, $H \in G$, and an integer $d \geq 1$ such that
\[ d_{\chi} = \sum_{H \in G} \mathbf{Ind}_H^{\cG}(\chi_H).\]
\end{enumerate}
and since the family of cyclic subgroups of $\cG$ satisfies the first of these properties we have that each character is a linear combination with rational coefficients of the characters induced by characters of cyclic subgroups. 
\end{theorem}

It is desirable to have an $\ZZ$ linear combination of characters of cyclic groups as apposed to a $\QQ$-linear combination as this has applications to Artin L-functions. This can in fact be achieved. By the construction of the pullback defining surface orders, we may compute characteristic polynomials for all factors of the Artin L-functions, including the ramified primes. Moreover, as the characteristic polynomials are invariant under base change, if we are able to completely understand the polynomial invariants and obtain an explicit description of them, the geometry of the rings of invariants, and a detailed description of how this is related to the representation theory of the Galois group and the algebraic group, then we are able to obtain a much deeper understanding of the Artin L-functions. Further, if we are able to understand the \emph{"semi-invariants"}, i.e. polynomial invariants under the action of special linear groups (see for example \cite{H1} IV \S 11.4), then we can completely understand moduli stacks of the automorphic representations. This will be completed in a forthcoming paper on the invariant theory of surface algebras. In particular one can show the following:

\begin{theorem}\label{Main Theorem}
\begin{enumerate}
\item The rings of polynomial semi-invariants (under an action of special linear groups), and therefore of the polynomial invariants under arbitrary base change are all semi-group rings and are the coordinate rings of affine toric varieties. 

\item Moreover the parametrizing varieties of representations, i.e. the "representation varieties," of fixed dimension for a given surface algebra are normal, Cohen-Macaulay, and have rational singularities. 

\item By results of \cite{CCKW} one may deduce that the moduli spaces of the semi-stable "regular modules" are in fact isomorphic to projective lines. It turns out that the invariant rings (as apposed to semi-invariant) provide such moduli spaces and the Artin L-functions are defined in terms of these invariants.
\end{enumerate}
\end{theorem}

The interested read can refer to \cite{MS} pg. 197, \cite{SV}, \cite{LP}, \cite{C}, and \cite{D} for the background theory necessary to prove this. In particular, one must generalize the methods of \cite{LP}, similar to what is presented in \cite{Do1, Do2, Do3, Do4} to the case of surface algebras. We may also treat the field extensions which are \emph{not} Galois using the decomposition on pg. 5 of \cite{LP} and the generalizations given in \cite{Do1, Do2, Do3, Do4}. In particular, we have the following:

\begin{corollary}\label{Main Corollary Artin L-functions}
Fix any character $\chi$ of the group $G(E/F)$. The Artin L-functions for any finite field extension $K/F$ of number fields can be realized as 
\[ L(s, \chi) = m_1^{a_1}m_2^{a_2} \cdots m_r^{a_r} \]
a product of monomials corresponding to characters of cyclic subgroups of the automorphism group of the field extension, corresponding to the local orders of a surface order. Moreover, we have that
\[ \chi = a_1\chi_1 + a_2\chi_2 + \cdots + a_r\chi_r \]
where the $\chi_i$ are characters of the cyclic subgroups of $G(E/F)$ corresponding to the local orders and $a_i \in \ZZ$. 
\end{corollary}

Some of the properties of the corresponding semigroup rings have been recently observed in \cite{Ci1}, \cite{CN1, CN2, CN3} and \cite{N1, N2, N3}, for \emph{Galois extensions} $K/\QQ$. It is stated there that Artin's conjecture is equivalent to various other statements, for example those given in Corollary 1.7 of \cite{Ci1}. One of the equivalent conditions given there is
\[ k[H(s_0)] = k[x_1, x_2, ..., x_r] \]
for all $s_0 \in \CC \backslash \{0\}$, where $\CC \subset k \subset \mathcal{M}_{<1}$, $\mathcal{M}_{<1}$ the field of meromorphic functions of order $<1$, $H(s_0)$ is the semigroup corresponding to Artin L-functions holomorphic at $s_0 \in \CC \backslash \{0\}$, and $k[x_1, ..., x_r]$ is the affine polynomial ring. In other words, if $\mathbf{Proj} \left(k[x_1, x_2, ..., x_r]\right) \cong \PP_K^r$. By results of \cite{CCKW}, this is always true for rational invariants corresponding to "regular components" of module varieties, given by a sum of "band modules". 

With a complete understanding of the polynomial (semi)invariants, one can verify many of these conditions explicitely now since an explicit description of the generators and relations of the invariant rings and semi-invariant rings can be computed. This can be done for any Artin L-functions, not just those corresponding to automorphic representations. Moreover, the case of arbitrary extensions of number fields (not necessarily Galois) can be treated in detail, and it can be shown that the equations for induced characters are $\ZZ$-linear. All of these properties may be deduced for example from from \cite{SV}, \cite{D}, \cite{DW}, \cite{C}, and \cite{Do1, Do2, Do3, Do4}. Moreover, the behavior of the \emph{"restricted partition functions"} mentioned in \cite{CN1, CN2, CN3} can be explained via the results of \cite{Do1, Do2, Do3, Do4}. 

Setting up the background material on representation varieties, Schofield semi-invariants, the equivalent determinantal semi-invariants, and the Geometric Invariant Theory needed to describe the moduli stacks here would take us too far from the material presented here, so this is deferred to the a followup paper, which uses methods of \cite{AS2} as well as methods developed in \cite{C}, \cite{CW}, \cite{CC}, and \cite{CCKW} to describe the indecomposable representations and the geometry of the representation varieties, and provides a complete description of all polynomial (semi)invariant functions on the representation varieties under the action of a connected reductive algebraic group whose Lie algebra is given by the noncommutative normalization of the surface order given by the completion of the surface algebra. 

\section{Reflections and Comments}

The fibre product of matrix algebras introduced in this paper are essentially examples of those introduced by Igor Burdan and Yuri Drozd and their study of \emph{nodal curves}.

\textcolor{darkred}{\hrule}

\end{document}